\documentclass[11pt,a4paper,english]{article}

\usepackage{amsmath}
\usepackage{latexsym}
\usepackage[english]{babel}
\usepackage{amsfonts}
\usepackage[T1]{fontenc}
\usepackage[latin1]{inputenc}
\usepackage{amsthm}
\usepackage{amssymb}
\usepackage{version}
\usepackage{fancyhdr}
\usepackage{caption}
\usepackage{stmaryrd}
\usepackage{url}
\usepackage{hyperref}

\def \F {{\mathcal F}}

\def \H {{\mathcal H}}

\def \R {{\mathbb R}}

\def \vr {\varrho}
\def \vrb {\overline{\varrho}}
\def \cadlag {c\`{a}dl\`{a}g}

\newtheorem{theorem}{Theorem}[section]

\newtheorem{corollary}[theorem]{Corollary}
\newtheorem{definition}[theorem]{Definition}
\newtheorem{rem}[theorem]{Remark}
\newtheorem{proposition}[theorem]{Proposition}

\newtheorem{ass}[theorem]{Assumption}

\newcommand{\ud}{\mathrm d}
\newcommand{\ds}{\displaystyle}

\numberwithin{equation}{section}

\newcommand{\mail}[1]{\href{mailto:#1}{\texttt{#1}}}

\newcommand{\ind}{\mbox{1 \hspace{-9 pt} I}}

\begin{document}

\title{The Zakai equation of nonlinear filtering for jump-diffusion observations: existence and uniqueness}
\author{Claudia Ceci \footnote{Department of Economics,
University of Chieti-Pescara, Viale Pindaro 42,
I-65127 Pescara, Italy
Email: \mail{ceci@sci.unich.it}.}
\and Katia Colaneri \footnote{Department of Economics,
University of Chieti-Pescara, Viale Pindaro 42,
I-65127 Pescara, Italy
Email: \mail{colaneri@sci.unich.it }.}}

\maketitle

\begin{abstract}
This paper is concerned with the nonlinear filtering problem for a general Markovian partially observed system $(X,Y)$, whose dynamics is modeled by correlated jump-diffusions having common jump times. At any time $t\in[0,T]$, the $\sigma$-algebra ${\cal F}^Y_t:= \sigma\{ Y_s: s\leq t\}$ provides all the available information about the signal $X_t$. The central goal of stochastic filtering is to characterize the filter, $\pi_t$, which is  the conditional distribution of  $X_t$, given the observed data ${\cal F}^Y_t$. It has been proved in \cite{cc2011}  that $\pi$ is the unique probability measure-valued process satisfying a nonlinear stochastic equation, the so-called Kushner-Stratonovich equation (KS-equation).  In  this paper the aim is to describe the filter $\pi$ in terms of the unnormalized filter $\vr$, which is  solution to  a linear stochastic differential equation, the so-called  Zakai equation.  We prove equivalence between strong uniqueness for the solution to the Kushner Stratonovich equation and strong uniqueness for the solution to the Zakai one and, as a consequence,  we deduce pathwise uniqueness for the solutions to the Zakai equation by applying the Filtered Martingale Problem approach (\cite{KO, cc2011}). To conclude, some particular cases are discussed.
\end{abstract}


Amsprimary: 93E11; 60J75; 60J60

\section{Introduction}

The objective of stochastic filtering is to find the best estimate, in some sense, of the state process $X$ of a stochastic dynamical system $(X,Y)$, from a partial observation described by a  process $Y$. This subject found several applications throughout the years, which includes a great variety of engineering problems, the study of the global climate, the estimation of the economy condition, the identification of tumours using digital recording.
\medskip

The literature concerning nonlinear filtering is quite rich; within the first results we mention \cite{Z} and \cite{Ku}. We can distinguish three main scenarios based on different dynamics of the observation process. In \cite{Ka, LiS, KO} the process $Y$ gives observation of $X$ in additional Gaussian noise, in \cite{Br, KKM, FR2, CG1, CG5, C} it is studied the case of counting process or  marked point process observations and more recently the case of mixed type observations (marked point processes and diffusions) has been taken into account in \cite{FR3, FS, CC, FSX, cc2011}. We want to focus on this last case which is the same considered in this note. In \cite{FR3, FS, CC} the information flow has the structure ${\cal F}^m_t \vee {\cal F}^\eta_t$, where $m(\ud t,\ud z)$ is a marked point process whose dynamics is influenced by a stochastic factor $X$ and $\eta$ gives observations of $X$ in additional Gaussian noise. These particular structures of the observation have a financial motivation; nevertheless in a general framework, it may be  meaningful to consider the case where the observation flow is generated by a jump-diffusion process as in the model developed in \cite{cc2011} and also in this paper. For this model, in \cite{cc2011}, the filtering problem has been studied using the innovation approach and the filter $\pi$ (defined by $\pi_t(f):=\mathbb{E}[f(t,X_t)|\mathcal{F}^Y_t]$, for some suitable functions $f$)  has been characterized as the unique solution to the Kushner-Stratonovich equation (or KS-equation). The equation found is not only infinite dimensional but also non linear, and it has a complicated structure that makes it is not suitable for computation. On the other hand, the problem can be faced from a different point on view, that of the Zakai equation for the unnormalized filter. This equation still remains infinite dimensional, but it has the advantage to be linear, that is the reason why it is well suited to be analyzed under numerical approximations, such as, for example the Galerkin method (see \cite{FSX}, \cite{GP1} and \cite{GP2}) or the optimal quantization approach \cite{GPPP}. In particular, in  \cite{FSX}, the Galerkin method is applied to  a mixed type observation, given by the pair $(\eta, \widetilde{N})$ where  $\widetilde{N}$ is a compensated Poisson process with unobservable intensity  and  $\eta$ gives observations of $X$ in additional Gaussian noise.

\medskip

In this paper we want to study the filtering problem from this second point of view of the Zakai equation. To be more precise, we will look for an appropriate Girsanov change of probability measure on $(\Omega, {\cal F}^Y_T, P |_{\mathcal{F}^Y_T})$, where $T$ is some fixed time horizon, which leads to an  equivalent probability measure $P_0$, defined by $\ds \left.\frac{\ud P_0}{\ud P} \right |_{\mathcal{F}^Y_T}=Z_T$, such that  the associated unnormalized filter, $\vr_t(\ud x):= Z_t^{-1} \pi_t(\ud x)$, solves  a linear equation of Zakai's  type.  Moreover, under additional hypotheses the new probabilty measure $P_0$ defined on $(\Omega, {\cal F}^Y_T)$ coincides with the restriction on ${\cal F}^Y_T$ of a probability measure equivalent to P on ${\cal F}_T$, which allows us to establish the analogy between our approach and the classical one based on the Kallianpur-Striebel formula.

\medskip

In order to represent the filter in terms of the unnormalized one, we need to prove uniqueness of the solutions to the Zakai equation. We show the equivalence between strong uniqueness for the KS-equation and strong uniqueness for the Zakai equation. Then, we deduce pathwise uniqueness for the solutions to the Zakai equation by  pathwise uniqueness results for the solutions to the KS-equation proved in \cite{cc2011} by applying the Filtered Martingale Problem approach. This method was introduced in \cite{KO} and then generalized in \cite{KN}. In both of the papers it is applied to prove strong uniqueness for both of the equations, KS and Zakai, in frameworks of signals observed in Gaussian white noise. Here, we extend the uniqueness result for the Zakai equation in the more general case of jump-diffusion observations. To the authors' knowledge this is the first time that the dynamics of the unnormalized filter is computed in the case of a partially observed system $(X,Y)$, where the signal $X$ and the observation $Y$ are described by correlated jump diffusion processes having common jump times.

\medskip

The paper is organized as follows. The filtering model is described in Section 2. In Section 3 we derive the Zakai equation of the nonlinear filtering problem. Technical difficulties introduced by working with real valued random counting measures instead of counting processes brought us to make the assumption that there exists a transition function $\eta(t,y,\ud z)$ such that the $\mathcal{F}^Y_t$-predictable measure $\eta(t,Y_{t^-},\ud z)$ is  equivalent to $\F^Y_t$-dual predictable projection of the random counting measure $m(\ud t, \ud z)$, associated with the jumps of the process $Y$, $\pi_{t^-}(\lambda \phi(\ud z))$, where $\pi_{t^-}$ denotes the left version of the filter. In Section 4 we discuss uniqueness for the solutions of the Zakai equation. In Section 5, we conclude by giving some examples where pathwise uniqueness for the solutions to the Zakai equation is fulfilled. In particular we analyze three models. In the first  the observation process is given by a jump diffusion with jump sizes in a finite set; in the second one we consider the case where the observation dynamics is driven by independent point processes with unobservable intensities; in the last one we assume that the state process $X$ is pure jump process taking values in a countable space. For the first two examples we compute explicitly the measure $\eta(t,Y_{t^-},\ud z)$ which ensure us that existence and uniqueness for the Zakai equation hold. Instead in the third one we derive directly, by a recursive procedure, uniqueness for the solution to the Zakai.

\section{The partially observed model and preliminary results}

Throughout the paper, we consider a partially observed system $(X,Y)$, on a complete filtered probability space  $(\Omega, \{{\cal F}_t\} _{t \in [0,T]}, P)$, where $T >0$ is some fixed time horizon. The dynamics of the system is described by the following pair of stochastic differential equations

\begin{equation}\label{sistema}
\left\{\begin{split}
\ud X_t=&\;b_0(t,X_t)\ud t+\sigma_0(t,X_t)\ud W^0_t+\int_Z   K_0(t,X_{t^-};\zeta)N(\ud t,\ud \zeta);\qquad X_0=x_0\in \mathbb{R}\vspace{1em}\\
\ud Y_t=&\;b_1(t,X_t,Y_t)\ud t+\sigma_1(t,Y_t)\ud W^1_t+\int_Z   K_1(t,X_{t^-},Y_{t^-};\zeta)   N(\ud t,\ud \zeta);\qquad Y_0=y_0\in\mathbb{R}
\end{split}\right.
\end{equation}

where $W^0$ and $W^1$ are two correlated Brownian motions with correlation coefficient $\rho\in [-1,1]$ and $N(\ud t,\ud \zeta)$ is a Poisson random measure on $\mathbb{R}^+\times Z$ whose intensity $\nu(\ud \zeta)\ud t$ is a $\sigma-$finite measure on a measurable space $(Z,\mathcal{Z})$.
\medskip

In this model, $X$ represents a signal, also called the state process, which cannot be directly observed and the process $Y$, described by a correlated process having common jump times with $X$, gives the observation.
\medskip

The coefficients $b_0(t,x)$, $b_1(t,x,y)$, $\sigma_0(t,x) >0, \sigma_1(t,y)>0$, $K_0(t,x;\zeta)$ and $K_1(t,x,y;\zeta)$
are $\mathbb{R}$-valued measurable functions of their arguments. As in \cite {cc2011} we assume strong existence and uniqueness for the solutions of the system $(\ref{sistema})$. Sufficient conditions are summarized by Assumption $\ref{assumption_c}$ in Appendix C.

\medskip
At any time $t$ the $\sigma$- algebra ${\cal F}^Y_t := \sigma \{ Y_s : s\leq t\}$ provides all the available information about the signal $X_t$. Our aim is to characterize the conditional distribution of $X_t$ given ${\cal F}^Y_t$, that represents the most detailed description of our knowledge of  $X_t$.

\medskip

In order to describes the jump component of $Y$ we introduce the integer-valued random measure

\begin{equation} \label{m} m(\ud t, \ud z) =  \sum_{s: \Delta Y_s \neq 0}
  \delta _{\{s, \Delta Y_s\}} (\ud t, \ud z)
\end{equation}
where $\delta_a$ denotes the Dirac measure at  point $a$. Note that the following equality holds
\begin{equation}\label{misura_m}
\int_0^t\int_{\mathbb{R}}z\;m(\ud s,\ud z)=\int_0^t\int_{Z}K_1(s,\zeta)N(\ud s,\ud \zeta)
\end{equation}
and, in general, for any measurable function $g:\mathbb{R}\rightarrow \mathbb{R}$
\begin{equation}\label{integrale_rispetto_m}
\int_0^t\int_{\mathbb{R}}g(z)\;m(\ud s,\ud z)=\int_0^t\int_{Z}\ind_{\{K_1(s,\zeta)\neq 0\}}g\left(K_1(s,\zeta)\right)N(\ud s,\ud \zeta).
\end{equation}

\medskip

For all $t\in[0,T]$, for all $A\in\mathcal{B}(\mathbb{R})$, we define

 $$d^0(t,x):=\{\zeta\in Z :K_0(t,x;\zeta)\neq0\},\hspace{3em}
                    d^1(t,x,y):=\{\zeta\in Z:K_1(t,x,y;\zeta)\neq0\},$$

                  \begin{equation} \label{d} d^A(t,x,y):=\{\zeta\in Z : K_1(t,x,y;\zeta)\in A\smallsetminus\{0\}\}\subseteq d^1(t,x,y), \end{equation}
             and, finally,
  \begin{equation} \label{insiemi} D^A_t = d^A(t,X_{t^-},Y_{t^-}) \subseteq D_t = d^1(t,X_{t^-},Y_{t^-}), \quad D^0_t = d^0(t,X_{t^-}). \end{equation}

Normally  $D^0_t \cap D_t  \neq\emptyset\;\;P-a.s.$ and this models the fact that state process and observation  may have common jump times.
 \medskip

In the sequel we will write  $b_i(t),  \sigma_i(t), K_i(t , \zeta)$, $i=0,1,$ for
 $b_0(t,X_t), b_1(t,X_t,Y_t)$, $\sigma_0(t,X_t), \sigma_1(t,Y_t)$, $K_0(t,X_{t^-};\zeta)$ and $K_1(t,X_{t^-},Y_{t^-};\zeta)$
   respectively and
 we will assume the following requirements

 \begin{ass}\label{hp_ks}

\begin{equation*}
\mathbb{E}  \int_0^T \int_{Z}|K_i(t, \zeta)| \nu(\ud \zeta)  \ud t  <  \infty, \quad \mathbb{E}  \int_0^T |b_i(t)| \ud t    < \infty,  \quad  \mathbb{E} \int_0^T \sigma_i^2(t) \ud t  <  \infty \quad i=0,1;
\end{equation*}

\begin{equation}\label{hp_ks2}
\mathbb{E}\int_0^T\nu(D^0_t \cup D_t)\ud t<\infty.
\end{equation}

\end{ass}

Note that under these constraints the pair $(X,Y)$ is a Markov process and both of the processes $X$ and $Y$ have finite first moment.
\medskip


As proved in Proposition 2.2 of \cite{C}, the $(P,{\cal F}_t)$-dual predictable projection, $m^p(\ud t,\ud z)$, of  $m(\ud t,\ud z)$ (see  \cite{JS, Br} for the definition), can be written as

 \begin{equation} \label{cara} m^p(\ud t,\ud z) = \lambda_t \phi_t(\ud z) \ud t, \end{equation}

where $\phi_t(\ud z)$ is a probability measure over $(\mathbb{R},\mathcal{B}(\mathbb{R}))$ and $\forall A \in \mathcal{B}(\mathbb{R})$

\begin{equation} \label{mp}
m^p(\ud t,A) = \lambda_t \phi_t(A) \ud t = \nu( D^A_t) \ud t.
\end{equation}

Define the functions $\lambda(t,x,y):=\nu(d^1(t,x,y))$ and $\phi(t,x,y,\ud z):=\int_{d^1(t,x,y)}\delta_{K_1(t,x,y;\zeta)}(\ud z) \nu(\ud \zeta)$, then the $\ds (P,\mathcal{F}_t)$-local characteristics of the integer valued counting measure $m(\ud t, \ud z)$, given by
\begin{equation}\label{caratt.locali}
(\lambda_t, \phi_t(\ud z))=(\lambda(t, X_{t^-},Y_{t^-}), \phi(t, X_{t^-},Y_{t^-},\ud z)),
\end{equation}
depend on the state process, and therefore they are not directly observable. In particular, $\ds \forall A \in \mathcal{B}(\mathbb{R})$, $\lambda_t\phi_t(A)=\nu(D^A_t)$ is the $(P,\mathcal{F}_t)-$intensity of the point process $N_t(A)=m((0,t]\times A)$ that counts the jumps of the process $Y$ until time $t$ whose widths belong to $A$ and $ \lambda_t = \nu( D_t)$ provides the $(P, {\cal F}_t)$-predictable intensity of the point process $N_t = m((0,t]\times \mathbb{R} )$ which counts the total number of jumps of $Y$ until  $t$.
\medskip

Similarly, define the function $\lambda^0(t,x):= \nu(d^0(t,x))$, the process $ \lambda^0_t:=\lambda^0(t,X_{t^-})=\nu(D^0_t )$ furnishes the $(P, {\cal F}_t)$-predictable intensity of the point process $N^0_t $ which counts the total number of jumps of $X$ until time $t$. Condition $(\ref{hp_ks2})$ imply that the processes $N$ and $N^0$ are both non-explosive and integrable (\cite{Br}).

\medskip

Let us introduce the filter defined as
\begin{equation} \label{filter}
\pi_t(f) : =  \mathbb{E} [ f(t,X_t) | {\cal F}^Y_t ]
\end{equation}

 for any measurable function $f(t,x)$ such that $ \mathbb{E} |f(t,X_t) |< \infty$ $\forall t\in [0,T]$.  It is known  that $\pi$ is a probability measure-valued process with c\`{a}dl\`{a}g trajectories (see \cite{KO}). We denote by $\pi_{t^-}$ his left version. In particular, for all functions $F(t,x,y)$ such that $ \mathbb{E} |F(t,X_t, Y_t) |< \infty$ (resp. $ \mathbb{E} |F(t,X_{t^-}, Y_{t^-}) |< \infty$) $\forall t\in [0,T]$, we will use the notation

 \begin{equation*}
 \pi_t(F):=\pi_t(F(t,\cdot,Y_t))\qquad \Big(\textrm{resp.}\quad \pi_{t^-}(F):=\pi_{t^-}(F(t,\cdot,Y_{t^-}))\Big).
 \end{equation*}

\medskip

\begin{rem}
We recall that for any $\F_t$-progressively measurable process $\psi$, satisfying the inequality $\mathbb{E}  \int_0^T|\psi_t|  \ud t < \infty$, the process
$\mathbb{E} [  \int_0^T \psi_t  \ud t  | {\cal F}^Y_t ] - \int_0^T \pi_t(\psi)  \ud t$   is a $(P,\mathcal{F}^Y_t)$-martingale. In particular, this implies that

\begin{equation}\label{vecchia}
 \mathbb{E}  \int_0^T| \pi_t (\psi) |  \ud t =  \mathbb{E}  \int_0^T |\psi_t|  \ud t  < \infty.
 \end{equation}
 \end{rem}

Denote by  $\nu^p(\ud t, \ud z)$  the $(P,\mathcal{F}^Y_t)$-predictable projection of the integer-valued measure $m(\ud t,\ud z)$; the following proposition, proved in \cite{C}, gives a representation of $\nu^p(\ud t, \ud z)$  in terms of the filter.

\medskip

\begin{proposition}\label{proiezione_predicibile}

The $(P,\mathcal{F}^Y_t)$-predictable projection of the integer-valued measure $m(\ud t,\ud z)$ is given by

\begin{equation}\label{basta}
 \nu ^p(\ud t,\ud z) =    \pi_{t^-} (\lambda\phi(\ud z))\ud t,
\end{equation}

that is, for  any $A \in \mathcal{B}(\mathbb{R})$

\begin{equation} \label{filtra}\nu ^p((0,t] \times A) =
 \int_0^t  \pi_{s^-}(\lambda \phi(A)) \ud s=   \int_0^t  \pi_{s ^-}\big (\nu(d^A(., Y_{s^-}) ) \big)\ud s .\end{equation}

\end{proposition}
\medskip

The measure
\begin{equation}
m^\pi( \ud t, \ud z )= m(\ud t, \ud z)-\pi_{t^-}(\lambda \phi (\ud z))\ud t
\end{equation}

is called the $\mathcal{F}^Y_t$-compensated martingale measure and has the property that for all $\F^Y_t$-predictable process indexed by $z$, $H(t,z)$ satisfying
\[
\mathbb{E}\left[\int_0^T\int_{\mathbb{R}}H(t,z)\pi_{t^-}(\lambda\phi(\ud z))\ud t\right]<\infty \qquad \left(\textrm{resp.}\quad \int_0^T\int_{\mathbb{R}}H(t,z)\pi_{t^-}(\lambda\phi(\ud z))\ud t<\infty \quad P-a.s.\right),
\]
the process $\ds \int_0^t\int_{\mathbb{R}}H(s,z)m^{\pi}(\ud s, \ud z)$ is a $(P, \F^Y_t)$-martingale (resp. local martingale).

\bigskip

Finally, assume that

\begin{equation}
 \mathbb{E} \int_0^T \left|\frac{b_1(t)}{\sigma_1(t)}\right|\ud t\ <\infty\quad \textrm{and}\quad  \mathbb{E} \int_0^T \pi_t^2\left|\frac{b_1}{\sigma_1}\right|\ud t <\infty , \label{hp_deboli1}
 \end{equation}

then we can define the so called innovation process $I$, which, in our framework, is given by

\begin{equation} \label{inn}
I_t:=W^1_t + \int_0^t \left \{ \frac{b_1(s)}{\sigma_1(s)}-\pi_{s}\left(\frac{b_1}{\sigma_1}\right)\right \}\ud s.
\end{equation}

It is not difficult to verify that the process $I$ is a $(P,\mathcal{F}^Y_t)$-Brownian motion.

\begin{rem}
Let us notice that, by Jensen's inequality and $(\ref{vecchia})$,  the condition
\begin{equation}
 \mathbb{E}  \int_0^T\left| \frac{b_1(t) }{\sigma_1(t) } \right| ^2 \ud t\ <\infty, \label{hp_forte1}
 \end{equation}
 which is usually required in the classical approach, implies $( \ref{hp_deboli1})$.

\end{rem}

\medskip

The process $I$ and the $\mathcal{F}^Y_t$-compensated martingale measure $m^\pi $, play a central role in describing the dynamics of the filter.
More precisely,  in \cite{cc2011}, under Assumption $\ref{hp_ks}$, $(\ref{hp_forte1})$ and assuming that the process
\begin{equation} \label{MG}
L_t=\mathcal{E}\left\{-\int_0^t  \frac{b_1(s)}{\sigma_1(s)} \ud W^1_s\right\}\
\end{equation}
 is a $(P, {\cal F}_t)$-martingale  ($\mathcal{E}$ denotes the Dol\'{e}ans-Dade exponential),
it is proved that the filter is  a solution to the Kushner-Stratonovich equation driven by $I$ and $m^\pi $. This result can be improved and the theorem stated below gives the same thesis under weaker hypotheses. In particular, we replace condition $(\ref{hp_forte1})$ with  $(\ref{hp_deboli1})$ and the assumption that  $L$ is   a $(P, {\cal F}_t)$-martingale with  the hypothesis that
\begin{equation} \label{MG1}
  \widehat {L} _t=\mathcal{E}\left\{-\int_0^t  \pi_s \left(\frac{b_1}{\sigma_1}\right)\ud I_s\right\}\
\end{equation}
is a $(P, {\cal F}^Y_t)$-martingale.

\begin{rem}
Observe that if $L$ is   a $(P, {\cal F}_t)$-martingale  then $ \widehat{L}$ is a $(P, {\cal F}^Y_t)$-martingale. In fact, if we define the probability measure  $\widetilde{Q}_0$  equivalent to $P$ over   ${\cal F}_T$,
such that

 \begin{equation} \label{Q1}
 \left.{\ud \widetilde{Q}_0 \over \ud P}\right|_{\mathcal{F}_T} = L_T,
\end{equation}

then by the Girsanov Theorem, the process
\begin{equation}\label{girsanov}
\widetilde{W}^1_t  := W_t^1+\int_0^t \frac{b_1(s)}{\sigma_1(s)}\ud s
\end{equation}

 is a $(\widetilde{Q}_0, {\cal F}_t)$-Wiener process. Taking into account $(\ref{inn})$, we get
 \begin{equation}\label{wiener}
 \widetilde{W}^1_t = I_t + \int_0^t \pi_s  \left( \frac{b_1}{\sigma_1} \right)\ud s,
 \end{equation}
 which implies that $\widetilde{W}^1$ is a $(\widetilde{Q}_0,  {\cal F}^Y_t)$-Wiener process.
Again by the Girsanov Theorem we deduce that
  \begin{equation} \label{Qg}
  \widehat {L} _t =  \left.{\ud \widetilde{Q}_0 \over \ud P}\right|_{{ \cal F}^Y_t }= \mathbb{E}  [ L_t | { \cal F}^Y_t ], \end{equation}

then  $\widehat {L}$ is  a $(P, {\cal F}^Y_t)$-martingale.
\end{rem}

 \begin{theorem}[The Kushner-Stratonovich equation] \label{EQ}
 Assume that Assumption $\ref{hp_ks}$ and $(\ref{hp_deboli1})$ hold and that  $\widehat L$ defined in $(\ref{MG1})$ is a $(P, {\cal F}^Y_t)$-martingale, then the filter $\pi$ solves the following Kushner-Stratonovich equation,  that is,  $\forall f \in C^{1.2}([0,T] \times \R)$
\begin{equation} \label{ks}
\pi_t (f) = f(0,x_0) + \int_0^t \pi_s(L^X f) \ud s +
\int_0^t\int_{\mathbb{R}}  w^\pi_s(f,z) m^\pi (\ud s, \ud z)   + \int_0^t   h^\pi_s(f)  \ud I_s   \end{equation}

where
\begin{equation} \label{d1}
 w^\pi_t (f,z) = {\ud\pi_{t^-} (\lambda\phi f) \over  \ud\pi_{t^-} (\lambda \phi )} (z) - \pi_{t^-}(f) +
  { \ud\pi_{t^-} (\overline{L} f) \over  \ud\pi_{t^-} \left(\lambda \phi \right)} (z), \end{equation}

 \begin{equation} \label{d2}
  h^\pi_t(f)= \sigma_1^{-1}(t)[\pi_{t}( b_1 f) - \pi_{t}(b_1) \pi_{t}(f)] + \rho  \pi_{t}\left(\sigma_0  {\partial f \over \partial x}\right).   \end{equation}

 Here by ${\ud\pi_{t^-} (\lambda\phi f) \over  \ud\pi_{t^-} (\lambda \phi )} (z) $ and
$  { \ud\pi_{t^-} (\overline{L} f) \over  \ud\pi_{t^-} \left(\lambda \phi \right)} (z)$ we mean the Radon-Nikodym derivatives of the measures $\pi_{t^-} (\lambda f \phi (\ud z))$ and $\pi_{t^-} (\overline{L} f)(\ud z)$, with respect to $\pi_{t^-} \left(\lambda \phi (\ud z) \right)$. The operator  $\bar L_tf$ defined by $\ds \bar L_t f (\ud z):= \bar L f(., Y_{t^-}, \ud z)$ and

\begin{equation}\label{operatoreL}
\forall A \in \mathcal{B}(\mathbb{R}) \quad  \bar L f(t,x,y,A) := \int_{d^A(t,x,y)} [ f(t, x + K_0(t,x;\zeta)) - f(t,x) ] \nu(\ud \zeta)
\end{equation}

takes into account common jump times between the signal $X$ and the observation $Y$.

Finally, the operator $L^X$ given by

\begin{equation*}
L^Xf(t,x)=\frac{\partial f}{\partial t}+ b_0(t,x) \frac{\partial f}{\partial x}+ \frac{1}{2}\sigma_0^2(t,x) \frac{\partial^2 f}{\partial x^2 }+ \int_Z \{ f(t, x + K_0(t,x;\zeta)) - f(t,x) \} \nu(\ud \zeta).
\end{equation*}
denotes  the generator of the Markov process $X$.

\end{theorem}

\begin{proof}

The proof is similar to that of Theorem 3.2 in \cite{cc2011}. We only need to observe that the $(P, {\cal F}^Y_t)$-martingale representation Theorem in terms of $I$ and $m^\pi$, proved in Proposition 2.6 in \cite{cc2011}, still holds true even if we replace the condition $(\ref{hp_forte1})$ with  $(\ref{hp_deboli1})$ and the assumption that the process $L$,  defined by $(\ref{MG})$, is   a $(P, {\cal F}_t)$-martingale with  the hypothesis that $\widehat {L}$ given in $(\ref{MG1})$ is a $(P, {\cal F}^Y_t)$-martingale.
In fact, it is sufficient introduce  the probability measure  ${Q}_0$ on  $(\Omega, {\cal F}^Y_T)$, equivalent to the restriction of $P$ over ${\cal F}^Y_T$, defined as

 \begin{equation} \label{Q1}
 \left.{\ud Q_0 \over \ud P}\right|_{\mathcal{F}^Y_T} = \widehat {L}_T.
\end{equation}

By the Girsanov Theorem $I_t + \int_0^t \pi_s  \left( \frac{b_1}{\sigma_1} \right)\ud s$  is a $(Q_0, {\cal F}^Y_t)$-Wiener process and, taking into account $(\ref{inn})$, we obtain that
$ I_t + \int_0^t \pi_s  \left( \frac{b_1}{\sigma_1} \right)\ud s = W_t^1+\int_0^t \frac{b_1(s)}{\sigma_1(s)}\ud s = \widetilde{W}^1_t $.

We write $\mathcal{F}^m_t$ for the filtration generated by the random counting measure $m(\ud t,\ud z)$, then, since
\begin{equation*}
    \ud Y_t=\int_{\mathbb{R}}z\; m(\ud t,\ud z) +\sigma_1(t,Y_t) \ud \widetilde{W}^1_t,
\end{equation*}
as in Proposition 2.6 in \cite{cc2011}, we can deduce that $ \mathcal{F}^Y_t=\mathcal{F}^m_t\vee\mathcal{F}^{\widetilde{W}^1}_t$ and that every   $(P, \mathcal{F}^Y_t)$-local martingale $M_t$ admits  the following decomposition

\begin{equation} \label{rapp1}
M_t = M_0 + \int_0^t  \int_{\mathbb{R}} w(s,z) m^\pi (\ud s, \ud z) + \int_0^t  h(s) \ud I_s,
 \end{equation}
where $w(t,z)$ is an $\mathcal{F}^Y_t$-predictable process  and $h(t)$ is an $\mathcal{F}^Y_t$-adapted process  such that
$$ \int_0^T \int_{\mathbb{R}} \left|w(t,z)\right|   \pi_{t^-}(\lambda \phi(\ud z))\ud t < \infty,  \quad \int_0^T h(t) ^2 \ud t < \infty \quad P-a.s.$$

Finally,  as in  Theorem 3.2 in  \cite{cc2011},  by applying the innovation method, we can conclude that the filter $\pi$ solves the equation $(\ref{ks})$.
\end{proof}

Let us observe that the KS-equation is an infinite-dimensional and nonlinear stochastic differential equation and so, in general, it is difficult to handle. Then it can be useful to characterize the filter in terms of a simpler equation. For doing so we will determine a probability measure $P_0$ over $(\Omega, {\cal F}^Y_T)$, equivalent to the restriction of $P$ onto $ {\cal F}^Y_T$, defined by

\begin{equation} \label{Den} \left.\frac{\ud P_0}{\ud P}\right|_{\mathcal{F}^Y_t}=Z_t, \end{equation}

where $Z$ is a suitable strictly positive $(P, {\cal F}^Y_t)$-martingale, chosen in such a way that the so-called unnormalized filter
$\vr$, defined by
\begin{equation}
\varrho_t(\ud x):= Z^{-1}_t \pi_t(\ud x)
 \end{equation}
 satisfies a linear stochastic differential equation, the Zakai equation.

\begin{rem}
Note that, if the measure $P_0$ is the restriction of a probability measure $\widetilde{P}_0$ equivalent to $P$ over the whole filtration $\mathcal{F}_T$  then the unnormalized filter can be written as
$$\ds \varrho_t(f):= \mathbb{E}^{\widetilde{P}_0}\left[f(t,X_t) \widetilde{Z}_t^{-1}|\mathcal{F}^Y_t\right], $$
where   $ \widetilde{Z}_t  := \ds \left.\frac{\ud \widetilde{P}_0}{\ud P}\right|_{\mathcal{F}_t}$. This follows from the well known  Kallianpur-Striebel formula

\begin{equation}\label{kall_stri}
\pi_t(f)=\frac{\mathbb{E}^{\widetilde{P}_0}\left[f(t,X_t) \widetilde{Z}_t^{-1}|\mathcal{F}^Y_t\right]}{\mathbb{E}^{\widetilde{P}_0}\left[  \widetilde{Z}_t^{-1}|\mathcal{F}^Y_t\right]},
\end{equation}

since $\mathbb{E}^{\widetilde{P}_0}\left[  \widetilde{Z}_t^{-1}|\mathcal{F}^Y_t\right] =  Z^{-1}_t $.  In order to derive the Zakai equation under mild conditions we do not require the existence of  such a  probability measure $\widetilde{P}_0$ defined on $(\Omega, \mathcal{F}_T)$, as in the classical reference probability method, but we will work directly with the probability measure $P_0$ defined on $(\Omega, {\cal F}^Y_T)$.
\end{rem}

The first step is to mention a complete version of the Girsanov Theorem to be applied on the model considered in this note.

\subsection{Girsanov change of probability}

\begin{theorem}\label{Girsanov}

Let $\varphi(t)$ and $\psi(t,z)$ be two processes $\mathcal{F}^Y_t$-adapted and  $\mathcal{F}^Y_t$-predictable respectively such that

\begin{equation}
 \int_0^T |\varphi(t)|^2 \ud t < \infty , \quad \int_0^T  \int_{\mathbb{R}} |\psi( t , z) | \pi_t(\lambda \phi (\ud z)) \ud t < \infty \quad P-{\rm a.s.}\label{altre1}
 \end{equation}
 \begin{equation}
1+\int_{\mathbb{R}}\psi(t,z) m(\{t\},\ud z)>0  \quad P-a.s.\;\;\; \forall t\in[0,T].\label{altre2}
\end{equation}
Define the process $L$ as
 \begin{equation} \label{density}
\ud L_t=L_{t^-}\left[\varphi(t)\ud I_t+\int_{\mathbb{R}} \psi( t, z) \left( m(\ud t,\ud z) -\pi_{t^-}(\lambda \phi (\ud z)) \ud t  \right)  \right].
\end{equation}
$L$ is a $(P,\mathcal{F}^Y_t)$-strictly positive local martingale.
If more
\begin{equation}
\mathbf E[ L_T ] =1 \label{martingala},
\end{equation}
 $L$ is a strictly positive $(P,\mathcal{F}^Y_t)$-martingale.

\noindent
Then, under $(\ref{martingala})$,  there exists a probability measure $Q$ defined on $(\Omega,\mathcal{F}^Y_T)$, equivalent to the restriction of  $P$ over $\mathcal{F}^Y_T$, such that

$$\left.\frac{ \ud  Q}{ \ud P}\right|_{\mathcal{F}^Y_t} = L_t,$$
and
\begin{description}
\item[(i)]
the process $\ds W^{Q}_t:= I_t - \int_0^t\varphi(s)\ud s$
is a $(Q, \mathcal{F}^Y_t)$-Brownian motion
\item[(ii)] the $(Q, \mathcal{F}^Y_t)$-dual predictable projection of the integer-valued measure $m(\ud t, \ud z)$ is
$$
\nu^{Q}(\ud z)\ud t=(1+\psi( t, z)) \pi_{t^-}(\lambda\phi (\ud z)) \ud t.
$$
\end{description}

\end{theorem}



It may be useful to investigate on whether the condition $(\ref{martingala})$
is satisfied, that is,  under which hypotheses $L$ is a strictly positive martingale. For the diffusive case the Novikov criterium provides a sufficient condition, for the most general case, there exists a similar criterium less known in literature, due to  Protter and Shimbo (Theorem 9 in \cite{ps2008}), that we mention below.

\begin{theorem}\label{thm9_protter}
Let $M$ be a locally square integrable martingale such that $\Delta M >-1$. If
\begin{equation}\label{protter_shimbo}
\mathbb{E}\left[\exp\left\{ \frac{1}{2}\langle M^c,M^c \rangle_{T}  +  \langle M^d,M^d  \rangle_{T}     \right\}\right]<\infty,
\end{equation}
where $M^c$ and $M^d$ are the continuous and the purely discontinuous martingale parts of $M$, then $\mathcal{E}(M)$ is a martingale on $[0,T]$, where $T$ can be $\infty$.
\end{theorem}

The following corollary translates the theorem we have just stated in our setting.

\begin{corollary}
Let $\varphi(t)$ and $\psi(t,x)$ be two processes $\mathcal{F}^Y_t$-adapted and $\mathcal{F}^Y_t$-predictable respectively, satisfying (\ref{altre2}). Assume that

\begin{equation}\label{shimbo_3}
\mathbb{E}\left[       e^{  \int_0^T     \left\{       \frac{1}{2}  |\varphi(t)|^2 + \int_{\mathbb{R}} |\psi (t,z)|^2   \pi_t(\lambda \phi (\ud z))    \right\} \ud t }\right]<\infty,
\end{equation}

then the process $L$
is a $(P,\mathcal{F}^Y_t)$-martingale on $[0,T]$.

\end{corollary}

In the sequel we refer to $(\ref{shimbo_3})$ as the Protter-Shimbo condition.

\begin{proof}
Let $M$ be given by
\begin{equation}\label{martingala_1}
\ud M_t=\varphi(t)\ud I_t + \int_{\mathbb{R}} \psi(t,z) \left[ m(\ud t,\ud z)-\pi_{t^-}(\lambda \phi(\ud z))\ud t  \right].
\end{equation}
If $(\ref{altre2})$ and $(\ref{shimbo_3})$ hold true, then $\ds \int_0^T\left|\varphi(t)\right|^2\ud t < \infty $ and $\int_0^T|\psi(t,z)|^2\pi_{t}(\lambda \phi(\ud z))\ud t<\infty$ $P-a.s.$ and $M$ is a $(P,\mathcal{F}^Y)$-locally square integrable martingale such that $\Delta M_t>-1$ $P-a.s.$ $\forall t \in [0,T]$, having sharp brackets
\begin{equation*}
\langle M^c,M^c\rangle_T=\int_0^T|\varphi(t)|^2 \ud t  \quad \textrm{and} \quad \langle M^d, M^d \rangle_T = \int_0^T\int_{\mathbb{R}}|\psi(t,z)|^2\pi_{t}(\lambda \phi(\ud z))\ud t.
\end{equation*}
Then the hypothesis $(\ref{shimbo_3})$ translates exactly $(\ref{protter_shimbo})$ in the jump-diffusion case. Therefore, by Theorem $\ref{thm9_protter}$ we get the claimed result.

\end{proof}

\begin{rem}
By applying the Cauchy-Schwarz inequality, we are able to split the assumption $(\ref{shimbo_3})$ in two separated sufficient conditions on the continuous part and on the purely discontinuous part of the martingale $M$ written in $(\ref{martingala_1})$. Indeed, since
\begin{equation*}
\mathbb{E}\left[e^{ \int_0^T \left(\frac{1}{2} |\varphi(t)|^2 +\int_{\mathbb{R}}|\psi(t,z)|^2 \pi_t(\lambda \phi(\ud z))\right)\ud t }\right]\leq \mathbb{E}\left[e^{ \int_0^T |\varphi(t)|^2 \ud t }\right]^{\frac{1}{2}}\mathbb{E}\left[e^{  2\int_0^T\int_{\mathbb{R}}|\psi(t,z)|^2\pi_t(\lambda \phi(\ud z))\ud t}\right]^{\frac{1}{2}},
\end{equation*}
it can be clearly deduced that, if

$$ \mathbb{E}\left[e^{ \int_0^T |\varphi(t)|^2 \ud t }\right]<\infty, \quad  \mathbb{E}\left[e^{  2\int_0^T\int_{\mathbb{R}}|\psi(t,z)|^2\pi_t(\lambda \phi (\ud z))\ud t}\right]<\infty$$

then $(\ref{shimbo_3})$ is fulfilled.

\end{rem}


\section{The Zakai equation} \label{zak}

We assume that Assumption $\ref{hp_ks}$ and $(\ref{hp_deboli1})$ hold, and, in order to perform a suitable Girsanov change measure on $(\Omega, \mathcal{F}^Y_T)$, we make the following additional hypothesis.

\medskip
\begin{ass}\label{assumption_eta}
Assume that there exists a transition function $\eta(t,y,\ud z)$ such that the $\mathcal{F}^Y_t$-predictable measure $\eta(t,Y_{t^-},\ud z)$ is  equivalent to $\pi_{t^-}(\lambda \phi(\ud z))$ and
\begin{equation}\label{integrabilita_eta}
\mathbb{E}   \left[     \int_0^T\eta(t,Y_{t^-},\mathbb{R})\ud t     \right]     <\infty.
\end{equation}
\end{ass}

This means that there exists an $\mathcal{F}^Y_t$-predicatble process $\Psi(t,z)$ such that
\begin{equation}\label{defn_eta}
\pi_{t^-}(\lambda\phi(\ud z))\ud t=(1+\Psi(t,z))\eta(t,Y_{t^-},\ud z)\quad \textrm{and} \quad 1+\Psi(t,z)>0 \quad P-a.s.
\end{equation}

\medskip
Now, we want to introduce a probability measure denoted by $P_0$, defined on $(\Omega, \mathcal{F}^Y_T)$, which is equivalent to the restriction of $P$ over $\mathcal{F}^Y_T$, given by
\begin{equation}\label{p0}
\left.\frac{\ud P_0}{\ud P}\right|_{\mathcal{F}^Y_t}=Z_t=Z_t^0 Z^1_t
\end{equation}
where the processes $Z^0$ and $Z^1$ are described by the following dynamics

\begin{eqnarray}
 Z^0_t\!\!\!\!&=&\!\!\!\!\mathcal{E}\left\{-\int_0^t \pi_s\left(\frac{b_1}{\sigma_1}\right)\ud I_s\right\},\label{num2}\\
 Z^1_t\!\!\!\!&=&\!\!\!\!\mathcal{E}\left(\int_0^t\int_{\mathbb{R}}\left\{\frac{1}{1+\Psi(t,z)}-1\right\}(m(\ud t,\ud z)-\pi_{t^-}(\lambda \phi(\ud z))\ud t)\right).\label{num3}
\end{eqnarray}

As usual $\mathcal{E}$ denotes the Dol\'{e}ans-Dade exponential.  Let us observe that $Z$ is a strictly positive $(P,\mathcal{F}^Y_t)$-local martingale, nevertheless, if we want to define the probability measure $P_0$
via the equation $(\ref{p0})$ we make the following requirement.
\begin{ass}\label{martingala_Z}

Assume that $Z$ is a $(P,\mathcal{F}^Y_t)$-martingale.
\end{ass}

\begin{rem}
 Setting $\ds U(t,z)=\frac{1}{1+\Psi(t,z)}-1$, by Theorem $\ref{thm9_protter}$, a sufficient condition, which implies that Assumption $\ref{martingala_Z}$ is fulfilled, is given by

\begin{eqnarray}\label{shimbo1}
\mathbb{E}\left[ \! \exp\left \{\frac{1}{2}\int_0^T \!\! \pi_s^2\left(\frac{b_1}{\sigma_1}\right)\ud s+\int_0^T \!\!\! \int_{\mathbb{R}}U^2(s,z)\pi_s(\lambda \phi(\ud z))\ud s\right\}\right]<\infty.
\end{eqnarray}
Moreover, let us observe that if the following conditions hold
\begin{gather*}
\left|  \pi_t\left(  \frac{b_1}{\sigma_1} \right)   \right|  \leq  C_1,   \quad     |U(t,z)|\leq C_2 \quad P-a.s. \quad \forall t\in[0,T] \quad \forall z\in \mathbb{R},\\
 \quad \int_0^T \pi_t(\lambda      ) \; \ud t \leq C_3 \quad P-a.s.,
\end{gather*}
for $C_i,  i=1,2,3$ positive constants, then $(\ref{shimbo1})$ is verified. Clearly, it is sufficient that the ratio $\ds \frac{b_1(t)}{\sigma_1(t)}$ as well as the $(P,\mathcal{F}_t)$-intensity, $\lambda_t$, of the point process $N_t$, and $U(t,z)$ are $P-a.s.$-bounded processes to make these conditions true.


\end{rem}

As a consequence of the Girsanov Theorem,

\begin{equation}\label{browniano}
\widetilde{W}^1_t:= I_t + \int_0^t \pi_{s}\left(\frac{b_1}{\sigma_1}\right)\ud s
\end{equation}
is a $(P_0,\mathcal{F}^Y_t)$-Brownian motion and the $(P_0, \mathcal{F}^Y_t)$-predictable projection of the integer counting measure $m(\ud t, \ud z)$ is given by
\begin{equation}
 \eta(t, Y_{t^-}, \ud z)\ud t.
 \end{equation}
 In the sequel we will write $\eta_t(\ud z)$ for the measure $\eta(t,Y_{t^-}, \ud z)$, unless it is necessary to underline the dependence on the process.

 \medskip

 Note that, in particular, $N$ turns to be a $(P_0, \mathcal{F}^Y_t)$-non explosive point process, with intensity $\eta_t(\mathbb{R})$.

\medskip

We denote by $\vr$ the unnormalized filter associated with the measure $P_0$, defined by
\begin{equation}\label{defn_rho}
\vr_t(\ud x)={Z_t}^{-1} \pi_t(\ud x).
\end{equation}
The process $\vr$ is a finite measure valued \cadlag  \hskip 1mm    process. In particular,

\begin{equation}\label{rho1}
 \vr_t(1) := {Z_t}^{-1} =\left.\frac{\ud P}{\ud P_0}\right|_{\mathcal{F}^Y_t},
\end{equation}

whose dynamics, written in the next proposition, can be easily obtained by considering the effects of the Girsanov change of measure on the processes involved.

\begin{proposition}

Under the Assumption $\ref{martingala_Z}$, the process $\vr(1)$ solves

 \begin{equation}\label{din_rho1}
 \ud \vr_t(1)=\vr_{t^-}(1)\left\{\pi_t\left(\frac{b_1}{\sigma_1}\right)\ud \widetilde{W}^1_t  +    \int_{\mathbb{R}}\Psi(t,z) \left[m(\ud t, \ud z)-\eta_t ( \ud z)\ud t\right]\right\}.
 \end{equation}

\end{proposition}

\begin{rem}

In other terms $\vr(1)$ is the Dol\'{e}ans Dade exponential of the $(P_0,\mathcal{F}^Y_t)$-martingale
\begin{equation*}
\ud M_t=\pi_t\left(\frac{b_1}{\sigma_1}\right)\ud \widetilde{W}^1_t  +    \int_{\mathbb{R}}\ \Psi(t,z) \left[m(\ud t, \ud z)-\eta_t( \ud z)\ud t\right].
\end{equation*}

\end{rem}

We are now in the position  to derive a Zakai's type equation for the unnormalized filter.


\begin{theorem}[The Zakai Equation]\label{thm_zakai}
Under Assumption $\ref{assumption_eta}$ and Assumption $\ref{martingala_Z}$ , let $P_0$ be the probability measure defined by $(\ref{p0})$, $(\ref{num2})$, $(\ref{num3})$.
Assume the hypotheses of Theorem $\ref{EQ}$. If

\begin{equation}\label{hp_deboli2}
\int_0^T \left\{ \vr_t^2\left|\frac{b_1}{\sigma_1}\right|+\vr_t^2(\sigma_0) \right\} \ud t<\infty \quad P_0-a.s.,
\end{equation}

then, $\forall f \in \mathcal{C}_b^{1,2}([0,T]\times \mathbb{R})$, the unnormalized filter $\vr$ satisfies the equation

\begin{gather}
\begin{aligned}
&\ud \vr_t(f) =  \vr_t(L^Xf)\ud t + \left\{\frac{\vr_t(b_1 f)}{\sigma_1(t)}+\rho\;\;\vr_t\left(\sigma_0\frac{\partial f}{\partial x}\right)\right\} \ud \widetilde{W}^1_t+ \vspace{.7em} \\
&+ \int_{\mathbb{R}}\left[{\ud \vr_{t^-} (\lambda\phi f) \over  \ud \eta_t } (z) - \vr_{t^-}(f) +
  { \ud \vr_{t^-} (\overline{L} f) \over  \ud \eta_t } (z)\right]\left(m(\ud t,\ud z)-\eta_t(\ud z)\ud t\right).
\end{aligned}\label{unnorm.filter}
\end{gather}
\end{theorem}

\begin{proof}

Since $\vr_t(f)=\pi_t(f)\vr_t(1)$, recalling that the filter $\pi_t(f)$ satisfies equation $(\ref{ks})$, by It\^o's formula we get

\begin{equation*}
\begin{split}
\ud\vr_{t}(f)= &\vr_{t}(1)\pi_t(L^X f)\ud t
           +  \vr_{t}(1)\left[\frac{\pi_t(b_1 f)-\pi_t(b_1)\pi_t(f)}{\sigma_1(t)}+\pi_t\left(\sigma_0 \frac{\partial f}{\partial x}\right) \rho \right]\left\{\ud \widetilde{W}^1_t-\pi_t\left(\frac{b_1}{\sigma_1}\right)\right\}\\
            & + \vr_{t^-}(1) \int_{\mathbb{R}}\left({\ud \pi_{t^-} (\lambda\phi f) \over  \ud \pi_{t^-} (\lambda \phi )} (z) - \pi_{t^-}(f) +
  { \ud \pi_{t^-} (\overline{L} f) \over  \ud \pi_{t^-} \left(\lambda \phi \right)} (z)\right)\left[m(\ud t,\ud z)- \pi_{t^-}(\lambda \phi (\ud z))\ud t\right]\\
           & + \vr_{t}(1)\pi_t(f)\frac{\pi_t(b_1)}{\sigma_1(t)} \ud \widetilde{W}^1_t + \vr_{t^-}(1)\pi_{t^-}(f)\int_{\mathbb{R}}\left(\frac{\ud \pi_{t^-}(\lambda \phi)}{\ud \eta_t}(z)-1\right)(m(\ud t,\ud z)-\eta_t( \ud z))\ud t\\
           & + \left\{\frac{1}{\sigma_1(t)}(\pi_t(b_1 f)-\pi_t( f)\pi_t(b_1))+\rho \pi_t\left(\sigma_0\frac{\partial f}{\partial x}\right)\right\}\vr_t(1)\frac{\pi_t(b_1)}{\sigma_1(t)}\ud t\\
           &+\int_{\mathbb{R}}\vr_{t^-}(1)\left(\frac{\ud \pi_{t^-}(\lambda \phi)}{\ud \eta_t}(z)-1\right)\left({\ud \pi_{t^-} (\lambda\phi f) \over  \ud \pi_{t^-} (\lambda \phi )} (z) - \pi_{t^-}(f) +
  { \ud \pi_{t^-} (\overline{L} f) \over  \ud \pi_{t^-} \left(\lambda \phi \right)} (z)\right) m(\ud t,\ud z)\\
           = & \vr_t(L^Xf)\ud t + \left\{\frac{\vr_t(b_1 f)}{\sigma_1(t)}+\rho\;\;\vr_t\left(\sigma_0\frac{\partial f}{\partial x}\right)\right\} \ud \widetilde{W}^1_t\\
           & + \int_{\mathbb{R}}\vr_{t^-}(1)\left[{\ud \pi_{t^-} (\lambda\phi f) \over  \ud \eta_t } (z) - \pi_{t^-}(f) +
  { \ud \pi_{t^-} (\overline{L} f) \over  \ud  \eta_t} (z)\right]\left(m(\ud t,\ud z)-\eta_t( \ud z)\ud t\right),
\end{split}
\end{equation*}
 which is equivalent to (\ref {unnorm.filter}).
\end{proof}

\begin{rem}
Let $\psi_t$ be a $\F_t$-progressively measurable process satisfying the inequality $\mathbb{E}  \int_0^T|\psi_t|  \ud t < \infty$, then

\begin{equation}\label{conti}
  \mathbb{E}^{P_0}  \left[ \int_0^T  \vr_t |\psi|  \ud t \right] = \int_0^T   \mathbb{E} \left[  Z_t \vr_t |\psi|  \right] \ud t  =
 \mathbb{E}  \int_0^T \pi_t |\psi|  \ud t =  \mathbb{E}  \int_0^T |\psi_t|  \ud t < \infty
  \end{equation}
 where the last equality follows by $(\ref{vecchia})$.
Hence, taking into account Assumption  $\ref{hp_ks}$, $(\ref{integrabilita_eta})$ $(\ref{hp_deboli2})$ and $(\ref{conti})$ we obtain
\begin{equation*}
\int_0^T\left\{ \vr_t\left(\frac{b_1}{\sigma_1} f\right)+\rho\; \vr_t\!\!  \left(    \sigma_0     \frac{\partial f}{\partial x}     \right)  \right\}^2\ud t \leq C_f \int_0^T \left \{  \vr^2_t \left|\frac{b_1}{\sigma_1}\right| +  \vr^2_t(\sigma_0 )  \right \} \ud t<\infty  \quad P_0-a.s.
\end{equation*}

\begin{equation*}
\int_0^T\   \vr_t | L^X f| \ud t \leq \widetilde {C}_f  \int_0^T\   \vr_t ( 1 + |b_0| + |\sigma_0|^2 + \nu(\lambda^0) )\ud t < \infty \quad P_0-a.s.
\end{equation*}

\begin{gather*}
\begin{aligned}[t]
&\int_{\mathbb{R}} \left | {\ud \vr_{t^-} (\lambda\phi f) \over  \ud \eta_t } (z) - \vr_{t^-}(f) +
  { \ud \vr_{t^-} (\overline{L} f) \over  \ud  \eta_t} (z)\right | \eta_t( \ud z)\ud t \leq {} \\
&\int_0^T \{ \vr_{t}|\lambda f| + \eta_t(\mathbb{R})\vr_{t}|f| +\vr_{t}( \overline{L} f )(\mathbb{R}) \} \ud t \leq  \hat {C}_f \|f\|\int_0^T ( \vr_t(1)\eta_t(\mathbb{R}) + \vr_t(\lambda))\ud t<\infty \quad P_0-a.s.
\end{aligned}
\end{gather*}
for some suitable positive constants $C_f,\widetilde{C}_f , \hat {C}_f$.
 As a consequence all terms in equation $(\ref{unnorm.filter})$ are well defined and the integrals with respect to the  $(P_0, \mathcal{F}^Y_t)$-Brownian motion $\widetilde{W}^1_t$ and to the $(P_0, \mathcal{F}^Y_t)$-compensated martingale measure $m(\ud t,\ud z)-\eta_t( \ud z)\ud t $, are $(P_0, \mathcal{F}^Y_t)$-local martingales.
  \end{rem}
The Zakai equation (\ref{unnorm.filter}) has a natural recursive structure and it is linear between two consecutive jump times. In fact, let $\{T_n\}_{n\in\mathbb{N}}$ be the increasing sequence of jump times of the observation process  then for $t\in [T_n, T_{n+1})$, and $t \leq T$, $\vr$ verifies

\begin{equation*}
\begin{split}
\vr_{t}(f)-\vr_{T_n}(f)=& \int_{T_n}^t\vr_s(L^X f)\ud s + \int_{T_n}^t\left(\frac{1}{\sigma_1(s)}\vr_s(b_1 f)+\rho \vr_s\left(\sigma_0 \frac{\partial f}{\partial x}\right)\right)\ud \widetilde{W}^1_s \\
                       &- \int_{T_n}^t\int_{\mathbb{R}}\left[\frac{\ud \vr_{s^-}(\lambda \phi f)}{\ud \eta_s}(z)-\vr_{s^-}(f) + \frac{\ud \vr_{s^-}(\overline{L}f)}{\ud \eta_s}(z)\right]\eta_s(\ud z)\ud s\\
                        =&\int_{T_n}^t\vr_s(L^X f)\ud s + \int_{T_n}^t\left(\vr_s\left(\frac{b_1}{\sigma_1} f\right)+\rho \vr_s\left(\sigma_0 \frac{\partial f}{\partial x}\right)\right)\ud \widetilde{W}^1_s \\
                        &- \int_{T_n}^t \left\{\vr_{s}(\lambda f)- \vr_{s}(f)\eta_{s}(\mathbb{R})+\vr_{s}(\widetilde{L}f)\right\}\ud s
\end{split}
\end{equation*}

 where $\ds \widetilde{L}_tf:=\overline{L}_tf(\mathbb{R})$ and $\bar{L}_tf(\ud z)$ is defined by $(\ref{operatoreL})$.
Moreover, if $t$ is a jump time, i.e. $t=T_n$,

\begin{equation*}
\vr_{T_n}(f)-\vr_{T_n^-}(f)= \frac{\ud \vr_{T_n^-}(\lambda \phi f)}{\ud \eta_{T_n}}(Z_n) - \vr_{T_n^-}(f)+\frac{\ud \vr_{T_n^-}(\overline{L}f)}{\ud \eta_{T_n}}(Z_n)
\end{equation*}

where $Z_n=Y_{T_n}-Y_{T_{n^-}}$.
Hence $\vr_{T_n} (f) $ is completely determined by the observed data $(T_n, Z_n)$ and by the knowledge of $\vr_t(f)$
  for all $t\in[T_{n-1}, T_n)$, since $\vr_{T_n^-} (f) = \lim_{t \to  T_n^-} \vr_t(f)$.
\medskip

 Finally, the Zakai equation (\ref{unnorm.filter}) can also be written in the following way
\begin{equation}\label{zakai_alternativa}
\begin{split}
\ud\vr_{t}(f)= & \left\{  \vr_t(L_0^Xf)-\vr_t(\lambda f)+\eta_t(\mathbb{R})\vr_t(f)  \right\}\ud t + \left\{\frac{\vr_t(b_1 f)}{\sigma_1(t)}+\rho\;\;\vr_t\left(\sigma_0\frac{\partial f}{\partial x}\right)\right\} \ud \widetilde{W}^1_t\\
           & + \int_{\mathbb{R}}\left\{\frac{\ud \vr_{t^-}(\lambda \phi f)}{\ud \eta_t}(z) - \vr_{t^-}(f) + \frac{\vr_{t^-} (\overline{L} f)}{\ud \eta_t}(z)\right\} m(\ud t,\ud z),
\end{split}
\end{equation}

where $L^X_0$ represents the operator
 \begin{equation}\label{L_0}
 \begin{split}
 L_0^X f(t,x,y) &= L^X f(t,x) - \bar L f(t,x,y, \mathbb{R})\\
                &={\partial f \over \partial t}  + b_0(t,x) {\partial f \over \partial x} + {1\over 2} \sigma_0^2(t,x) {\partial^2 f \over \partial x^2} +\int_{\{\zeta\in Z: K_1(t,x,y;\zeta)=0\}}\!\!\!\!\!\!\!\!\!\!\![f(t,x+K_0(t,\zeta))-f(t,x)]\nu(\ud \zeta).
 \end{split}
 \end{equation}


\begin{rem}
Similarly also the Kushner-Stratonovich equation solved by the filter $\pi$ has an equivalent expression in terms of the operator $L^X_0$, given by

\begin{gather*}
\begin{aligned}[t]
& \ud \pi_t (f) = \{ \pi_t(L_0^X f)  + \pi_t(f) \pi_t(\lambda)  - \pi_t( \lambda f)  \}    \ud t + \left\{  \frac{    \pi_t(b_1 f) -   \pi_t(f)   \pi_t(b_1)}{\sigma_1(t)}    +    \rho \pi_t   \left(    \sigma_0     \frac{\partial f}{\partial x}     \right)  \right\}   \ud I_t  +\\
&\qquad \qquad +   \int_{\mathbb{R}} \left\{ \frac{\ud \pi_{t^-}(\lambda \phi f)}{\ud \pi_{t^-}(\lambda \phi)}(z)- \pi_{t^-}(f)+ \frac{\ud \pi_{t^-}(\bar{L}f)}{\ud \pi_{t^-}(\lambda \phi)}(z) \right\} m(\ud t, \ud z).
\end{aligned}
\end{gather*}

\end{rem}
\bigskip

In the last part of this section  we discuss the analogies between our approach and the classical one based on the Kallianpur-Striebel formula. To be more precise, in the classical reference probability measure method the probability measure $P_0$ is  the restriction over $\mathcal{F}^Y_T$ of a probability measure $\widetilde{P}_0$ equivalent to $\widetilde{P}$ on the whole filtration $\mathcal{F}_T$. Below we give sufficient conditions under which the classical approach can be applied to our model.

\begin{proposition}\label{classico}
Assume that there exists a transition function $\eta(t,y, \ud z)$ such that the $\mathcal{F}^Y_t$-predictable measure $\eta(t, Y_{t^-}, \ud z)$ is equivalent to $\lambda_t\phi_t(\ud z)$  and
$$
\mathbb{E}  \left[   \int_0^T\eta(t,Y_{t^-},\mathbb{R}) \ud t    \right]<\infty.
$$


Let Assumption $\ref{hp_ks}$ prevail and assume $(\ref{hp_forte1})$.
Define the process
\begin{equation}
\widetilde{Z}_t:=\mathcal{E}\left(-\int_0^t \frac{b_1(s)}{\sigma_1(s)} \ud W^1_s + \int_0^t\int_{\mathbb{R}}\left(\frac{\ud \eta_s}{\ud \lambda_s \phi_s}(z)-1\right)     \left( m(\ud s,\ud z)- \lambda_s\phi_s(\ud z)\ud s\right)  \right)
\end{equation}

and assume that $\widetilde{Z}$ is a strictly positive $(P,\mathcal{F}_t)$-martingale.
Let $\widetilde{P}_0$ be the probability measure defined by $\ds \left.\frac{\ud \widetilde{P}_0}{\ud P}\right|_{\mathcal{F}_t}=\widetilde{Z}_t.$ Then the probability measure $P_0$ coincides with  the restriction of $\widetilde{P}_0 $ over $\mathcal{F}^Y_T$. 
\end{proposition}

\begin{proof}

Let us define $\widehat{Z}_t:=\mathbb{E}\left[\widetilde{Z}_t|\mathcal{F}^Y_t\right]$. $\widehat{Z}_t$ is the Dol\'{e}ans-Dade exponential of a $(P,\mathcal{F}^Y_t)$-martingale $V$. From the Martingale Representation Theorem (Proposition 2.6 in \cite{cc2011}), we deduce the structure of the process $V$, which is
\begin{equation}
 V_t=V_0+\int_0^t h^V_s\ud I_s +\int_0^t \int_{\mathbb{R}} w^V(s,z) \Big(m(\ud s,\ud z)- \pi_{s^-}(\lambda \phi (\ud z))\ud s \Big).
\end{equation}

Now we need to identify these processes $h^V_t$ and $w^V(t,z)$.

By the Girsanov Theorem,  $\ds \widetilde{W}_t^1=W_t^1+\int_0^t\frac{b_1(s)}{\sigma_1(s)}\ud s$ is a $(\widetilde{P}_0,\mathcal{F}_t)$-Brownian motion, and the $(\widetilde{P}_0,\mathcal{F}_t)$-predictable projection of the integer valued measure $m(\ud t, \ud z)$ is $\eta_t(\ud z) \ud t$. Observe that, because of the structure of the innovation process $I$, the process $\widetilde{W}^1$ can also be written as $\ds \widetilde{W}_t^1=I_t+\int_0^t\pi_s\left(\frac{b_1}{\sigma_1}\right)\ud s$, which implies that $\widetilde{W}^1$ is a $(\widetilde{P}_0,\mathcal{F}^Y_t)$- Brownian motion. So, concerning the continuous part, we have $\ds h^V_t=\pi_t\left(\frac{b_1}{\sigma_1}\right)$.

Note that the $(\widetilde{P}_0,\mathcal{F}^Y_t)$-dual predictable projection of $m(\ud t, \ud z)$ is $\eta(t,Y_{t^-},\ud z)$, since $\eta$ is already $\mathcal{F}^Y_t$-predictable; then there exists a $(\widetilde{P}_0,\mathcal{F}^Y_t)$-predictable process $\overline{\psi}(t,z)$ such that
$$
\eta(t, Y_{t^-},\ud z)\ud t=\left(1+\overline{\psi}(t,z)\right)\pi_{t^-}(\lambda \phi(\ud z))\ud t
$$
or, equivalently, $\ds \overline{\psi}(t,z)=\frac{\ud \eta_t}{\ud \pi_{t^-}(\lambda \phi )}(z)-1$, which means that  $w^V(t,z)=\overline{\psi}(t,z)$.

Then, the martingale $V$ assumes the form

\begin{equation*}
V_t=V_0+\int_0^t \pi_s\left(\frac{b_1}{\sigma_1}\right) \ud I_s  +\int_0^t \int_{\mathbb{R}} \left\{\frac{\ud \eta_s}{\ud \pi_{s^-}(\lambda \phi )}(z)-1\right\}\left[m(\ud s, \ud z)-\pi_{s^-}(\lambda \phi (\ud z))\ud s\right]
\end{equation*}

and $\widehat{Z}$ solves

\begin{equation}
\ud \widehat{Z}_t=\widehat{Z}_{t^-}\left(\pi_t\left(\frac{b_1}{\sigma_1}\right) \ud I_t  + \int_{\mathbb{R}} \left\{ \frac{ \ud \eta_t}{\ud \pi_{t^-}(\lambda\phi)}(z)-1 \right\}\left[m(\ud t, \ud z)-\pi_{t^-}(\lambda \phi (\ud z))\ud t\right]
\right).
\end{equation}
So $\widehat{Z}=Z$, which implies that $P_0$ is the restriction of the measure $\widetilde{P}_0$ over $\mathcal{F}^Y_T$.

\end{proof}

\begin{rem}
Under the assumption of Proposition $\ref{classico}$, there exists a measurable function $\widetilde \Psi (t,x,y,z)$ such that $\lambda_t\phi_t(\ud z)=\{1+\widetilde \Psi (t,X_{t^-}, Y_{t^-},z)\}\eta_t(\ud z)$. Hence by direct computations the Radon-Nikodym derivative $\ds \frac{\ud \vr_{t^-}(\lambda \phi f)}{\ud \eta_t}(z)$ that appears in $(\ref{zakai_alternativa})$ can be written as
\[
\frac{\ud \vr_{t^-}(\lambda \phi f)}{\ud \eta_t}(z)=\vr_{t^-}(f ) +\vr_{t^-}(f \widetilde \Psi (z)).
\]

\end{rem}


\section{Uniqueness}

In this section our intention is to show whether the Zakai equation uniquely characterizes the unnormalized filter. For doing so, we first prove the equivalence between pathwise uniqueness for the solution to the KS-equation and pathwise uniqueness for the solution to the Zakai one. Then we deduce pathwise uniqueness for the solution of the Zakai equation from uniqueness results for the KS-equation proved in \cite{cc2011}, applying the Filtered Martingale Problem approach.

\medskip

 Here we give the definition of strong solutions for the filtering equations, which is a slight modification of that given in \cite{cc2011}, since we replaced the condition $(\ref{hp_forte1})$ with $(\ref{hp_deboli1})$.

\begin{definition}
A strong solution to the Kushner-Stratonovich equation is an $\mathcal{F}^Y_t-$adapted c\`{a}dl\`{a}g $\mathcal{P}(\mathbb{R})-$valued process $\mu$ defined on the filtered probability space $\ds (\Omega, \{\mathcal{F}_t\}_{t\in[0,T]},P)$ such that
 \begin{equation}\label{integrabilita_b2}
 \int_0^T \left\{\mu_s(b_2)+\mu_s^2\left|\frac{b_1}{\sigma_1}\right|\right\} \ud s <\infty \quad P-a.s.
 \end{equation}
 where $\ds b_2(t,x) = |b_0(t,x)|+\sigma^2_0(t,x) + \lambda^0(t,x)+\lambda(t,x,Y_{t^-})$, and solving the Kushner-Stratonovich equation that is, $\forall f \in \mathcal{C}^{1,2,2}_b([0,T]\times\mathbb{R}\times\mathbb{R})$ and $\forall t\leq T$

\begin{equation}\label{mu}
\begin{split}
\ud \mu_t(f) =&  \left\{   \mu_t(L^X_0 f) - \mu_t(\lambda f) +  \mu_t(\lambda)   \mu_t(f)     \right\} \ud t   +    \left\{    \mu_t\left(\frac{b_1}{\sigma_1} f\right) -   \mu_t(f)   \mu_t\left(\frac{b_1}{\sigma_1}\right)  +    \rho \mu_t   \left(    \sigma_0     \frac{\partial f}{\partial x}        \right)  \right\}  \ud    I^{\mu}_t  \\
                     &+ \int_{\mathbb{R}} \left[\frac{\ud \mu_{t^-}(\lambda \phi f)}{\ud \mu_{t^-}(\lambda \phi)}(z) - \mu_{t^-}(f)   + \frac{\ud \mu_{t^-}(\overline{L}f)}{\ud \mu_{t^-}(\lambda \phi)}(z)\right] m(\ud t, \ud z),
\end{split}
\end{equation}

where $\ds \ud I^\mu_t=\ud W^1_t+\left\{\frac{b_1(t)}{\sigma_1(t)}-\mu_t\left(\frac{b_1}{\sigma_1}\right)\right\} \ud t$.
\end{definition}
\medskip

The condition $(\ref{integrabilita_b2})$ makes the integrals in $(\ref{mu})$ well defined, as the following estimations imply

\begin{align}
 \int_0^T \!\!\left\{     \mu_s\left(\frac{b_1}{\sigma_1} f\right) -   \mu_s(f)   \mu_s\left(\frac{b_1}{\sigma_1}\right)  +    \rho \mu_s   \left(    \sigma_0     \frac{\partial f}{\partial x}     \right)  \right\}^2\!\! \ud t \leq B_f\int_0^T\!\! \left\{\mu_t(b_2)+\mu^2_t\left|\frac{b_1}{\sigma_1}\right|\right\} \ud t< \infty \quad P-a.s. \label{stima2}
 \end{align}

\begin{align}
\int_0^T\left|\mu_t(L_0^Xf)- \mu_t(\lambda f) +  \mu_t(\lambda)   \mu_t(f) \right|\ud t \leq \widetilde{B}_f\int_0^T\left\{1+\left|\mu_t(b_0)\right|+\mu_t(\sigma^2_0)+\mu_t(\lambda^0)+\mu_t(\lambda)\right\} \ud t \nonumber\\
\leq \widetilde{B}_f\int_0^T\mu_t(b_2)\ud t<\infty \quad P-a.s.\label{stima3}
\end{align}
for some suitable positive constants $B_f$ and $\widetilde{B}_f$, and moreover the integrand with respect to the the measure $m(\ud t, \ud z)$ turns to be bounded by $4 \|f\|$.

\medskip
\begin{rem}
Observe that the process $\pi$ is a strong solution to the Kushner-Stratonovich equation since it solves the equation $(\ref{mu})$, where $I^\pi=I$, and the following estimate holds,
\begin{eqnarray}
\mathbb{E}\int_0^T   \left[    \pi_s(b_2)+\pi^2_s\left|\frac{b_1}{\sigma_1}\right|  \right] \ud s
=  \mathbb{E}\int_0^T   \left\{ b_2(t,X_t)+ \pi^2_s \left|\frac{b_1}{\sigma_1}\right|\right\} \ud s<\infty\label{stimab2}
\end{eqnarray}
because of  Assumption $\ref{hp_ks}$ and $(\ref{hp_deboli1})$.
\end{rem}


\medskip

According to the statement above, we give the definition for the Zakai equation.
Define $\mathcal{M}(\mathbb{R})$ the set of all nonnegative finite measures over $\mathbb{R}$.

\begin{definition}
A strong solution to the Zakai equation is a \cadlag, $\mathcal{F}^Y_t$-adapted process $\xi$ defined on the probability space $(\Omega, \{\mathcal{F}_t\}_{t\in [0,T]}, P_0)$ taking values in $\mathcal{M}(\mathbb{R})$, such that $\xi_{t^-}(\lambda \phi (\ud z))<<\eta(t,Y_{t^-},\ud z)$ and
 \begin{equation}\label{integrabilita_b2_zakai}
 \int_0^T\left\{\xi_s(\overline{b}_2)+\xi^2_s\left|\frac{b_1}{\sigma_1}\right|+\xi_s^2(\sigma_0)\right\}\ud s <\infty \quad P_0-a.s.
 \end{equation}
 where $\ds \overline{b}_2(t,x) = \eta_t(\mathbb{R})+ |b_0(t,x)|+\sigma^2_0(t,x) + \lambda^0(t,x)+ \lambda(t,x,Y_{t^-})$, and satisfying the Zakai equation, that is, $\forall f \in \mathcal{C}_b^{1,2}([0,T]\times \mathbb{R})$
\begin{equation}\label{num4}
\begin{split}
\ud \xi_t(f) =&   \left\{   \xi_t(L^X_0 f) - \xi_t(\lambda f) +  \eta_t(\mathbb{R}) \xi_t(f)     \right\}\ud t    +    \left\{  \frac{\xi_t(b_1 f)}{\sigma_1(t)} +\rho \hskip 0.5mm\xi_t   \left(    \sigma_0     \frac{\partial f}{\partial x}     \right)  \right\}  \ud    \widetilde{W}^1_t  \\
                      &  +  \int_{\mathbb{R}}\left\{ \frac{\ud \xi_{t^-}(\lambda \phi f)}{\ud \eta_t}(z) -\xi_{t^-}(f)+\frac{\ud \xi_{t^-}( \overline{L} f )}{\ud \eta_t}(z)   \right\} m(\ud t,\ud z).
\end{split}
\end{equation}

\end{definition}

Of course, taking into account the equivalent expression of (\ref{num4}) given by
 \begin{equation}\label{num5}
\begin{split}
\ud \xi_t(f) =&    \xi_t(L^X f) \ud t    +    \left\{  \frac{\xi_t(b_1 f)}{\sigma_1(t)} +\rho \hskip0.5mm \xi_t   \left(    \sigma_0     \frac{\partial f}{\partial x}     \right)  \right\}  \ud    \widetilde{W}^1_t  \\
                      &  +  \int_{\mathbb{R}}\left\{ \frac{\ud \xi_{t^-}(\lambda \phi f)}{\ud \eta_t}(z) -\xi_{t^-}(f)+\frac{\ud \xi_{t^-}( \overline{L} f )}{\ud \eta_t}(z)   \right\} \Big(m(\ud t,\ud z)-\eta_t(\ud z)\ud t\Big),
\end{split}
\end{equation}
 as before, the condition $(\ref{integrabilita_b2_zakai})$ makes the processes in the equation well defined and moreover the integrals with respect to $\widetilde{W}^1$ and the compensated measure $m(\ud t,\ud z)-\eta(t, Y_{t^-},\ud z)$ become $(P_0,\mathcal{F}^Y_t)$-martingales. Indeed,

\begin{gather*}
\int_0^T\left\{  \frac{1}{\sigma_1(t)} \xi_t(b_1 f)-\rho\; \xi_t \!\!  \left(    \sigma_0     \frac{\partial f}{\partial x}     \right)  \right\}^2\ud t \leq c_f \int_0^T \left\{\xi^2_t(\sigma_0) +\xi^2_t\left|\frac{b_1}{\sigma_1}\right|\right\} \ud t <\infty  \quad P_0-a.s.\\
\int_0^T\int_{\mathbb{R}}\left|\frac{ \ud \xi_{t}(\lambda \phi f)}{\ud \eta_t}(z) -\xi_{t^-}(f)+\frac{\ud \xi_{t}( \overline{L} f )}{\ud \eta_t}(z) \right|\eta_t(\ud z) \ud t \leq  \widetilde{c}_f \int_0^T \left\{ \xi_t(1)\eta_t(\mathbb{R}) +\xi_t(\lambda) \right\} \ud t<\infty \quad P_0-a.s.\\
\int_0^T \left| \xi_t(L^X f)\right|\ud t \leq \widehat{c}_f \int_0^T \xi_t(\overline{b}_2)\ud t <\infty \quad P_0-a.s.
\end{gather*}

for some suitable positive constants $c_f,\widetilde{c}_f$ and $\widehat{c}_f$.


\begin{rem}
Under the hypotheses of Theorem \ref{thm_zakai}, the unnormalized filter  $\vr$  is a strong solution to the Zakai equation. In fact it  solves $(\ref{zakai_alternativa})$, and  $(\ref{integrabilita_b2_zakai})$ comes from $(\ref{conti})$.

\end{rem}

\subsection{Equivalence of the filtering equation}

In this part of the section we prove two theorems. The first one shows that uniqueness for Zakai equation implies uniqueness for KS-equation, the second one, instead  provides the converse implication.

\begin{theorem}
Assume the hypotheses of Theorem $\ref{thm_zakai}$ and strong uniqueness for the solution to the Zakai equation. Let $\mu$ be a strong solution to Kushner-Stratonovich equation such that $\mu_{t^-}(\lambda \phi(\ud z))$ is equivalent to the measure $\pi_{t^-}(\lambda \phi(\ud z))$. Then $\mu_t=\pi_t$ $P-a.s.$ for all $t\in[0,T]$.
\end{theorem}

\begin{proof}
According to the definition, if $\mu$ is a strong solution to KS-equation, it verifies $(\ref{integrabilita_b2})$ and the equation $(\ref{mu})$. Then we define the process $\theta$ as the unique solution to the following stochastic differential equation

\begin{equation}\label{theta}
\ud \theta_t=\theta_{t^-}        \left\{    \mu_t\left(\frac{b_1}{\sigma_1}\right) \ud \widetilde{W}^1_t  +  \int_{\mathbb{R}}\left(\frac{\ud \mu_{t^-}(\lambda \phi)}{\ud \eta_t}(z) - 1    \right) \left( m( \ud t, \ud z) - \eta_t( \ud z) \ud t   \right)       \right\}.
\end{equation}
Observe that the Radon Nikodym derivative $\ds \frac{\ud \mu_{t^-}(\lambda \phi)}{\ud \eta_t}(z)$ is well defined; in fact the measures $\ds \mu_{t^-}(\lambda \phi (\ud z))$ and $\ds \eta(t,Y_{t^-}, \ud z)$ are equivalent since they are both equivalent to $ \pi_{t^-}(\lambda \phi (\ud z))$.
\vskip 3mm

Let $\vrb_t(f):=\theta_t \mu_t(f)$, by using It\^{o}'s formula we get

\begin{equation*}
\begin{split}
\ud \vrb_t(f) = & \theta_{t}        \left\{           \mu_t(L^X_0 f) - \mu_t(\lambda f) +  \mu_t(\lambda)   \mu_t(f)     \right\} \ud t +\\
                & + \theta_t \left\{  \frac{1}{\sigma_1(t)} \left[    \mu_t(b_1 f) -   \mu_t(f)   \mu_t(b_1)  \right]   +    \rho \mu_t   \left(    \sigma_0     \frac{\partial f}{\partial x}        \right)    \right\}     \left\{ \widetilde{W}^1_t-\mu_t\left(\frac{b_1}{\sigma_1}\right) \ud t \right\}  +\\
                &   + \theta_{t^-}  \int_{\mathbb{R}} \left\{ \frac{\ud \mu_{t^-}(\lambda \phi f)}{\ud \mu_{t^-}(\lambda \phi)}(z) -\mu_{t^-}(f) + \frac{\ud  \mu_{t^-}( \overline{L} f )}{\ud \mu_{t^-}(\lambda \phi)}(z)   \right\} m(\ud t,\ud z)   +  \theta_t \mu_t(f) \mu_t  \left( \frac{b_1}{\sigma_1}  \right) \ud \widetilde{W}^1_t +  \\
                & + \theta_{t^-}\mu_{t^-}(f) \int_{\mathbb{R}} \left(\frac{\ud   \mu_{t^-}(\lambda\phi)}{\ud \eta_t}(z)  - 1    \right)     \left( m(\ud t, \ud z) - \eta_t(\ud z) \ud t   \right)  + \\
                & +\theta_t    \mu_t    \left(    \frac{b_1}{\sigma_1}    \right)   \left\{    \frac{1}{\sigma_1(t)}   \left[    \mu_t  (b_1 f) -   \mu_t  (f)   \mu_t   (b_1)  \right]   +    \rho \mu_t   \left(    \sigma_0     \frac{\partial f}{\partial x}     \right)  \right\} \ud t+ \\
                & + \theta_{t^-}\int_{\mathbb{R}} \left(\frac{\ud   \mu_{t^-}(\lambda\phi)}{\ud \eta_t}(z)  - 1 \right)  \left( \frac{\ud \mu_{t^-}(\lambda \phi f)}{\ud \mu_{t^-}(\lambda \phi)}(z) -\mu_{t^-}(f) + \frac{\ud  \mu_{t^-}( \overline{L} f )}{\ud \mu_{t^-}(\lambda \phi)}(z)   \right) m(\ud t,\ud z),
\end{split}
\end{equation*}

that, computing all the sums, becomes

\begin{equation*}
\begin{split}
\ud \vrb_t(f) =&   \left\{   \vrb_t(L^X_0 f) - \vrb_t(\lambda f) +  \eta_t(\mathbb{R}) \vrb_t(f)     \right\} \ud t   +    \left\{ \vrb_t\left(\frac{b_1}{\sigma_1} f\right)-\rho \vrb_t   \left(    \sigma_0     \frac{\partial f}{\partial x}     \right)  \right\}  \ud    \widetilde{W}^1_t  \\
                      &  + \int_{\mathbb{R}} \left\{ \frac{\ud \vrb_{t^-}(\lambda \phi f)}{\ud \eta_t}(z) -\vrb_{t^-}(f)+\frac{\vrb_{t^-}( \overline{L} f )}{\ud \eta_t}(z)   \right\} m(\ud t, \ud z).
\end{split}
\end{equation*}

This means that $\vrb$ satisfies the Zakai equation. In order to prove that $\vrb$ is a strong solution to the Zakai equation we also need to verify that $(\ref{integrabilita_b2_zakai})$ holds.
\vskip 3mm

Let us introduce the stopping time $\alpha_n:=\inf\left\{   t>0 : \theta_t>n \right\}$, then, taking into account $(\ref{integrabilita_b2})$, we have

\begin{equation*}
\begin{split}
\int_0^{T \wedge \alpha_n}    \left\{ \vrb_t(\overline{b}_2)+\vrb^2_t\left|\frac{b_1}{\sigma_1}\right|+\vrb_t^2(\sigma_0) \right\}    \ud t = \int_0^{T \wedge \alpha_n}   \left\{    \theta_t \mu_t(\overline{b}_2)+ \theta_t^2 \mu^2_t\left|\frac{b_1}{\sigma_1}\right|+ \theta_t^2 \mu_t^2(\sigma_0)  \right\}    \ud t\leq\\
 n^2 \int_0^{T \wedge \alpha_n} \left\{ \mu_t(\overline{b}_2)+  \mu^2_t\left|\frac{b_1}{\sigma_1}\right|+  \mu_t^2(\sigma_0)\right\} \ud t<\infty;
\end{split}
\end{equation*}

hence  $\vrb_t$ is a strong solution to the Zakai equation for all $t \in [0,\alpha_n \wedge T)$.\vskip 3mm
Strong uniqueness for the solution to the Zakai equation implies that $\vrb_t\ind_{\{t<\alpha_n\wedge T\}}=\vr_t\ind_{\{t<\alpha_n\wedge T\}}$ $P_0-a.s.$, but, on the other hand, $P_0$ and $P$ are equivalent probability measures so uniqueness holds $P-a.s.$ too.

 Taking $n\to\infty$ we get $\vrb_t=\vr_t$ $\forall t \in [0,T]$, since $\alpha_n \xrightarrow[n\to\infty]{} \infty$. This follows from the fact that $\ds \sup_t \theta_t <\infty$ $P-a.s.$ as proved in Appendix \ref{appb}, Proposition $\ref{sup}$.
\vskip 3mm

In particular, since $\mu_t(1)=1$, $\vrb_t(1)=\vr_t(1)=\theta_t$\hskip 2mm $P-a.s.$, so $\ds \mu_t=\frac{\vrb_t}{\vrb_t(1)}=\frac{\vr_t}{\vr_t(1)}=\pi_t$ \hskip 2mm  $P-a.s$, and then strong uniqueness for KS-equation is fulfilled.
\end{proof}

\medskip

\begin{theorem}\label{zakai_uniq}
Assume the hypotheses of Theorem \ref{thm_zakai} and suppose that pathwise uniqueness for the solution to the Kushner-Stratonovich equation holds. Let $\xi$ be a strong solution to the Zakai equation.
Then $\xi_t=\vr_t$ $P-a.s.$ for all $t\in[0,T]$.

\end{theorem}

\begin{proof}
Let $\xi$ be a strong solution to the Zakai equation hence the process $\xi_t(1)$ solves

\begin{equation*}
\xi_t(1)= 1+ \int_0^t \frac{1}{\sigma_1(s)}\xi_s(b_1) \ud \widetilde{W}^1_s + \int_0^t\int_{\mathbb{R}} \left[     \frac{\ud \xi_{s^-}(\lambda \phi)}{\ud \eta_s}(z) -\xi_{s^-}(1)       \right]    \left( m(\ud s, \ud z)  - \eta_s(\ud z)\ud s   \right).
\end{equation*}

Since we do not know, a priori, if $\xi_t(1)$ is strictly positive, we define $\beta(\varepsilon):=\inf\{t\in[0,T] | \xi_t(1)<\varepsilon \}\wedge T$, with  $\varepsilon \in (0,1)$. 
Let  $\ds \mu_t(f):=\frac{\xi_t(f)}{\xi_t(1)}$ for all $t\in[0,\beta(\varepsilon))$.  By It\^{o}'s formula we get

\begin{equation}
\begin{split}
\ud \mu_t(f) =&  \left\{   \mu_t(L^X_0 f) - \mu_t(\lambda f) +  \mu_t(\lambda)   \mu_t(f)     \right\} \ud t   +    \left\{  \frac{1}{\sigma_1(t)} \left[    \mu_t(b_1 f) -   \mu_t(f)   \mu_t(b_1)  \right]   +    \rho \mu_t   \left(    \sigma_0     \frac{\partial f}{\partial x}        \right)  \right\}  \ud    I^{\mu}_t  \\
                     &+  \int_{\mathbb{R}}  \left\{\frac{\ud  \mu_{t^-}(\lambda \phi f)}{\ud \mu_{t^-}(\lambda \phi)}(z) -\mu_{t^-}(f) + \frac{\ud \mu_{t^-}( \overline{L} f )}{\ud \mu_{t^-}(\lambda \phi)}(z)  \right\}m( \ud t,\ud z),
\end{split}
\end{equation}

that is the process $\mu_t$ satisfies the KS-equation for all $t\in[0,\beta(\varepsilon))$. Moreover since
\begin{gather*}
\begin{aligned}
&\int_0^{\beta(\varepsilon)} \left\{ \mu_t(b_2)+\mu^2_t\left|\frac{b_1}{\sigma_1}   \right|  \right\}\ud t < \int_0^{\beta(\varepsilon)}\left\{\frac{\xi_t(\overline{b}_2)}{\xi_t(1)}+\frac{\xi^2_t\left|\frac{b_1}{\sigma_1}\right|}{\xi^2_t(1)}\right\}\ud t <{}\\
&\int_0^{\beta(\varepsilon)}\left\{\frac{\xi_t(\overline{b}_2)}{\varepsilon}+\frac{\xi^2_t\left|\frac{b_1}{\sigma_1}\right|}{\varepsilon^2}\right\}\ud t< \frac{1}{\varepsilon^2}\int_0^{\beta(\varepsilon)}\left\{\xi_t(\overline{b}_2)+\xi^2_t\left|\frac{b_1}{\sigma_1}\right|\right\}\ud t < \infty \quad P-a.s.,
\end{aligned}
\end{gather*}

 $\mu$ is a strong solution to the KS-equation on $[0,\beta(\varepsilon))$, and as a consequence  $\forall t\in[0,\beta(\varepsilon))$  $\mu_t=\pi_t$ $P-a.s.$\\

Define $\ds \gamma_t :=\frac{\xi_t(1)}{\vr_t(1)}$, again by  It\^{o}'s formula $\gamma$ solves
\begin{equation}
\begin{split}
\ud \gamma_t=& \gamma_t\pi_t\left(\frac{b_1}{\sigma_1}\right)   \left\{  \pi_t\left(\frac{b_1}{\sigma_2}\right)- \mu_t\left(\frac{b_1}{\sigma_1}\right)  \right\} \ud t +  \gamma_t \left\{  \mu_t\left(\frac{b_1}{\sigma_2}\right)- \pi_t\left(\frac{b_1}{\sigma_1}\right)  \right\} \ud \widetilde{W}^1_t+\\
                   &+ \gamma_t\left\{ \pi_t(\lambda)-\eta_t(\mathbb{R})-\mu_t(\lambda) +\eta_t(\mathbb{R}) \right\} \ud t + \gamma_{t^-}\int_{\mathbb{R}}\left(\frac{\ud \mu_{t^-}(\lambda \phi)}{\ud \eta_t}(z)\frac{\ud \eta_t}{\ud \pi_{t^-}(\lambda \phi)}(z)-1\right)m(\ud t, \ud z).
\end{split}
\end{equation}

Since  $\forall t\in[0,\beta(\varepsilon))$,  $\mu_t=\pi_t$ $P-a.s.$ then $\forall t\in[0,\beta(\varepsilon))$, $\gamma_t=\gamma_0=1$ $P-a.s$   and, equivalently, $\forall t\in[0,\beta(\varepsilon))$, $\xi_t(1)=\vr_t(1)$ $P-a.s.$.
\vskip 3mm
As a consequence,  $\forall t\in[0,\beta(\varepsilon))$, $\ds \xi_t(f)=\mu_t(f)\xi_t(1)=\pi_t(f)\vr_t(1)=\vr_t(f) \quad P-a.s.$, which in particular implies that $\beta(\varepsilon)=\inf\{t\in[0,T] | \xi_t(1)<\varepsilon \}  \wedge T   \geq \inf\{t\in[0,T] | \vr_t(1)<\varepsilon \}  \wedge      T$.
\vskip 3mm

Finally, let us observe that by Proposition III.3.5 in \cite{JS}, recalled in Appendix \ref{appc}, $\inf_{t\in[0,T]}\vr_t(1)>0$ .

Then $\beta(\varepsilon) =T $ for some $\varepsilon<\varepsilon_0(\omega)$ where $\varepsilon_0(\omega)=\inf_{t\in [0,T]}\vr_t(1)>0$, and this concludes the proof.

\end{proof}






\subsection{Uniqueness for the Zakai Equation}

Finally we deduce strong uniqueness for the Zakai equation from strong uniqueness of the KS-equation proved in \cite{cc2011}, using the Filtered Martingale Problem approach.


\begin{theorem}\label{unicita_zakai}
Let $(X,Y)$ be the partially observed system in $(\ref{sistema})$. Assume the hypotheses of Theorem \ref{thm_zakai}, $(\ref{assumption_c})$ in Appendix $\ref{appendixC}$, and one of the following conditions

\begin{equation}\label{g1}
\sup_{t,x}\lambda^0(t,x) +  \sup_{t,x,y}\lambda(t,x,y)<\infty
\end{equation}

or
\begin{equation}\label{g2}
\sup_{t,x,y} \int_Z \{ |K_0(t,x;\zeta)| + |K_1(t,x,y, \zeta) | \} \nu(\ud \zeta) <\infty.
\end{equation}

Let $\xi$ be a strong solution to the Zakai  equation such that
 $\xi_{t^-}(\lambda \phi(\ud z))\ud t$ and $\vr_{t^-}(\lambda \phi(\ud z))\ud t$ are equivalent measures over $[0,T]\times \mathbb{R}$, then $\xi_t=\vr_t$ $P-a.s.$ for all $t\leq T$.
\end{theorem}

\begin{proof}
Let $\ds \mu(f) :=\frac{\xi(f)}{\xi(1)}$. By the same steps of the proof of Theorem \ref{zakai_uniq}, $\mu$ is a strong solution to the Kushner-Stratonovich equation over $[0,\beta(1/n))$, where $\ds \beta(1/n):=\inf\{t\in[0,T] | \xi_t(1)< 1/n \}\wedge T$. Observe that, by definition, the equivalence between the measures $\xi_{t^-}(\lambda \phi (\ud z))\ud t$ and $\vr_{t^-}(\lambda \phi (\ud z))\ud t$ corresponds to the equivalence between the measure $\mu_{t^-}(\lambda \phi (\ud z))\ud t$ and $\pi_{t^-}(\lambda \phi (\ud z))\ud t$. Moreover, by Proposition 5.1 in \cite{cc2011}, $(\ref{assumption_c})$ and either $(\ref{g1})$ or $(\ref{g2})$ imply  uniqueness for the solutions to the filtered martingale problem associated with the generator of the pair $(X,Y)$. Then the thesis follows from the proof of Theorem 3.14 in \cite{cc2011}, replacing $\tau_n$ by $\tau_n\wedge \beta(1/n)$.
\end{proof}

\section{Particular models}

In the sequel we analyze three special cases under which existence and uniqueness for the solutions to the Zakai equation are fulfilled.
In particular, in the first two examples we verify that Assumption $\ref{assumption_eta}$ holds true and provide the explicit expression for the measure $\eta_t(\ud z)$. Instead, in the third, one we derive uniqueness for the solution to the Zakai equation by direct computations.

\subsection{Observation dynamics with jump sizes in a countable space}

Consider the case where $Y$ is a jump diffusion process, whose jump sizes take values in a discrete set $\H \subset \R$.
In such a situation, the integer valued random measure $m(\ud t, \ud z)$ can be written as
\begin{equation} \label{m}
m(\ud t, \ud z) = \sum_{h\in \H \setminus \{0\}} m(\ud t,\{h\})\ \delta_h(\ud z)
\end{equation}
where $\ds m(\ud t,\{h\}) =
\sum_{n\geq 1}
\ind_{\{Z_n=h\}}\  \delta _{T_n} (\ud t)\  \ind_{\{T_n < \infty\}}$, and we recall that  $\{T_n, Z_n\}_{n\in \mathbb{N}}$ denotes the sequence of jump times of the process $Y$ and the corresponding jump sizes.
\medskip

The $(P, \mathcal{F}_t)$-dual predictable projection of $m(\ud t, \ud z )$ is given by

\begin{equation} \label{mp}
m^p(\ud t,\ud z) = \lambda_t \phi_t(\ud z) \ud t =\sum_{h \in \H\setminus\{0\}}\lambda_t^h \delta_h(\ud z),
\end{equation}

where, as in the general case,  $\ds \lambda(t,x,y)=\nu(d^1(t,x,y))$, and $\lambda_t=\lambda(t, X_{t^-}, Y_{t^-})$ is the intensity of the process $N_t$ which counts the total number of jumps of the process $Y$. Observe that $\forall h \in \H\setminus\{0\}$, $\lambda_t^h=\lambda_t \phi_t(\{h\})$, is the $(P, \mathcal{F}_t)$-intensity of the point process  $N^h_t=m([0,t)\times \{h\})$ which counts the jumps of the process $Y$, with width $h$.
\medskip

On the other hand the $(P, \mathcal{F}^Y_t)$-dual predictable projection of $m(\ud t, \ud z)$ is
\begin{equation} \label{proiez_predic2}
\pi_{t^-} (\lambda \phi(\ud z)) \ud t =\sum_{h \in \H\setminus\{0\}}\pi_{t^-}(\lambda^h) \delta_h(\ud z),
\end{equation}
 and the Kushner-Stratonovich equation becomes
\begin{equation}\label{ks3}
\begin{aligned}[t]
&\pi_t(f)= f(0,x_0) + \int_0^t\pi_s(L^X f) \ud s +\int_0^t \Big \{ \pi_s\left(\frac{b_1}{\sigma_1}f\right)-\pi_s\left(\frac{b_1}{\sigma_1}\right)\pi_s(f)+\rho \pi_s\left(\sigma_0 \frac{\partial f}{\partial x}\right) \Big \} \ud I_t \\
&\qquad + \sum_{h\in \H \setminus\{0\}}\int_0^t (\pi_{s^-}(\lambda^h))^+
\left\{\pi_{s^-}(\lambda^h f)-\pi_{s^-}(\lambda^h)\pi_{s^-}(f)+ \pi_{s^-}(R^hf)  \right\}\Big( \ud N^h_s-\pi_{s^-}(\lambda^h)\ud s \Big)
\end{aligned}
\end{equation}
where $a^+ := {1\over  a}\ \ind_{a>0}$,
$$
R^hf(t, x, y) := \int_{d_1^h(t, x, y)}\Big[f(t, x +  K_0(t,x;\zeta)) - f(t , x)\Big]\ \nu(\ud \zeta),
$$
and $d_1^h(t,x,y):=\{\zeta\in Z: K_1(t,x,y; \zeta)=h\}$.

\begin{rem}
In the case where the observation is given by a pure jump process taking values in a countable space, the filtering problem for this model has been studied in \cite{CG5}. Here, the authors derived the Kushner-Stratonovich equation  (Proposition 3.1 in \cite{CG5}) and provide a linearization method of this equation which does not work in presence of a diffusive term, as in the model considered in this note.
\end{rem}

Strong uniqueness for the solution to the KS-equation \ref{ks3} can be deduced by that of the general model proved in \cite{cc2011}, as follows.

\begin{corollary}[Uniqueness for the KS-equation]\label{corollary_h}
Assume that $\forall h \in \H \setminus \{0\}$
\begin{equation}\label{pilambda_h-positiva}
\pi_t(\lambda^h)>0 \quad P-a.s. \quad \forall t \in[0,T]
\end{equation}and that either $(\ref{g1})$ or $(\ref{g2})$ holds. Let $\mu$ be a strong solution of the KS-equation $(\ref{ks3})$ such that $\forall h \in \H \setminus \{0\}$
\begin{equation}\label{mulambda_h-positiva}
\mu_t(\lambda^h)>0 \quad P-a.s. \quad \forall t \in[0,T].
\end{equation}
Then $\mu_t=\pi_t \;\;P-a.s.$ for all $t\leq T$.
\end{corollary}

\begin{proof}
It is sufficient to see that under  $(\ref{pilambda_h-positiva})$ and $(\ref{mulambda_h-positiva})$, the measures $\pi_t(\lambda \phi (\ud z)) \ud t$ and $\mu_t(\lambda \phi (\ud z)) \ud t$ are equivalent, since, in this case, there exist the Radon Nikodym derivatives
$$
\frac{\ud\mu_{t^-}(\lambda \phi)}{\ud\pi_{t^-}(\lambda \phi)}(z)=\frac{\sum_{h\in \H\setminus\{0\}} \delta_{h}(z)\mu_{t^-}(\lambda^h)}{\sum_{h \in \H \setminus \{0\} } \delta_h(z)\pi_{t^-}(\lambda^h)} = \sum_{h\in\H\setminus\{0\}} \ind_{\{z=h\}} \frac{\mu_{t^-}(\lambda^h)}{\pi_{t^-}(\lambda^h)}
$$
  and

$$
\frac{\ud\pi_{t^-}(\lambda \phi)}{d\mu_{t^-}(\lambda \phi)}(z)=\frac{\sum_{h\in \H\setminus\{0\}} \delta_{h}(z)\pi_{t^-}(\lambda^h)}{\sum_{h\in\H\setminus\{0\}}\delta_{h}(z)\mu_{t^-}(\lambda^h)}=\sum_{h\in\H\setminus\{0\}}\ind_{\{z=h\}}\frac{\pi_{t^-}(\lambda^h)}{\mu_{t^-}(\lambda^h)},
$$
and then apply Theorem $3.14$ in \cite{cc2011}.
\end{proof}

Let us see that under the hypothesis $(\ref{pilambda_h-positiva})$, Assumption $\ref{assumption_eta}$ is fulfilled since the measure $\pi_{t^-}(\lambda \phi(\ud z))$ described in $(\ref{proiez_predic2})$ is equivalent to the measure
\begin{equation}\label{eta}
\eta_t( \ud z):=\sum_{h\in \H \setminus \{0\}}\delta_h(\ud z)
\end{equation}
which is deterministic and satisfies the condition $(\ref{integrabilita_eta})$ if and only if $\H$ is a finite set.
Hence the probability measure $P_0$ on $(\Omega, \mathcal{F}^Y_T)$ is given by
\begin{equation}\label{p0_h}
\left.\frac{\ud P_0}{\ud P}\right|_{\mathcal{F}^Y_T}=Z_t=\mathcal{E}\left(-\int_0^t \pi_s\left(\frac{b_1}{\sigma_1}\right)\ud I_t + \sum_{h\in \H\setminus \{0\} } \int_0^t \left(\frac{1}{\pi_{s^-}(\lambda^h)}-1\right)[\ud N^h_s-\pi_{s^-}(\lambda^h)]\right),
\end{equation}

and as in the general case we assume that $Z$ is a $(P,\mathcal{F}^Y_t)$-martingale. Under  $P_0$, $\forall h \in \H \setminus \{0\}$, the point processes $N^h_t=m([0,t)\times \{h\})$ become standard Poisson processes.

\begin{rem}
For this model the Protter-Shimbo condition, which ensure us that $Z$ is a martingale, is

\begin{equation}\label{shimbo_2}
 \mathbb{E}\left[\exp   \left\{  \frac{1}{2}\int_0^T\pi_t^2\left(\frac{b_1}{\sigma_1}\right) \ud t+\int_0^T \sum_{h\in\H\setminus\{0\}} \frac{(1-\pi_t(\lambda^h))^2}{\pi_t(\lambda^h)}         \ud t\right\}  \right]<+\infty.
\end{equation}
Sufficient conditions are given by

\[
\left|\pi_t\left(\frac{b_1}{\sigma_1}\right)\right|\leq C_1 \quad \textrm{and} \quad 0\leq C_2\leq \pi_t(\lambda^h)\leq C_3 \quad P-a.s. \quad \forall t\in [0,T]
\]
for some suitable positive constants $C_1, C_2$ and $C_3$.
\end{rem}

Recall that the unnormalized filter is defined by $\ds \vr_t(\ud x)=Z_t^{-1} \pi_t(\ud x)$, then the dynamics of $\vr_t(1)=Z_t^{-1}$ is described by the following equation

\begin{equation}
\ud \vr_t(1)=\vr_{t^-}(1)\left(\pi_t\left(\frac{b_1}{\sigma_1}\right)\ud \widetilde{W}^1_t + \sum_{h\in \H \setminus \{0\}}(\pi_{t^-}(\lambda^h)-1)(\ud N^h_t-\ud t)\right).
\end{equation}

\begin{corollary} [the Zakai equation]
Let $(X,Y)$ be the partially observed system where in particular the integer valued measure in the dynamics of the observation process is given by $(\ref{m})$, and the set $\H$ is finite.
Assume the hypotheses of Theorem $\ref{EQ}$ and $(\ref{pilambda_h-positiva})$.
Let $P_0$ be the probability measure defined by $(\ref{p0_h})$ where $Z$ is a $(P,\mathcal{F}^Y_t)$-martingale. Assume moreover $(\ref{hp_deboli2})$;
then, $\forall f \in \mathcal{C}_b^{1,2}([0,T]\times \mathbb{R})$, the unnormalized filter $\vr$ satisfies the Zakai equation

\begin{equation}\label{zakai_h}
\ud \vr_t(f)=\vr_t(L^X f)+ \left\{\vr_t\left(\frac{b_1}{\sigma_1}f\right)+ \rho \vr_t\left(\sigma_0\frac{\partial f}{\partial x}\right)\right\}\ud \widetilde{W}^1_t +\sum_{h\in \H \setminus \{0\}}[\vr_{t^-}(\lambda^h f)-\vr_{t^-}(f)+\vr_{t^-}(R^h f)]\Big( \ud N^h_t-\ud t \Big).
\end{equation}

\end{corollary}

\begin{corollary}[Uniqueness for the Zakai equation]\label{zakai-uniq2}
Let $(X,Y)$ be the usual partially observed system $(\ref{sistema})$, where in particular the integer valued random measure in the dynamics of $Y$ is given by $(\ref{m})$ and the set $\H$ is finite.
Assume that, $\forall h \in \H \setminus\{0\}$,  $\vr_t(\lambda^h)>0 \quad P-a.s. \quad \forall t\in[0,T]$, $(\ref{assumption_c})$ and either $(\ref{g1})$ or $(\ref{g2})$. Let $\xi$ be a strong solution to the Zakai equation such that, $\forall h \in \H\setminus\{0\}$
\begin{equation}\label{positiva_h}\xi_{t^-}(\lambda^h)>0 \quad P-a.s. \quad \forall t\in[0,T].\end{equation}
Then $\xi_t(f)=\vr_t(f)$ $P_0-a.s.$ $\forall t\in[0,T]$.
\end{corollary}

\begin{proof}
It is sufficient to see that, by $(\ref{positiva_h})$ and the hypothesis that $\vr_t(\lambda^h)>0$ $P-a.s.$ $\forall h \in \H \setminus\{0\}$ and $\forall t \in [0,T]$, the measures $\ds \xi_{t^-}(\lambda \phi(\ud z))=\sum_{h\in \H\setminus \{0\}}\xi_{t^-}(\lambda^h) \delta_{h}(\ud z)$ and $\ds \vr_{t^-}(\lambda \phi(\ud z))=\sum_{h\in \H\setminus \{0\}}\vr_{t^-}(\lambda^h) \delta_{h}(\ud z)$ are equivalent. Then the result follows from Theorem $\ref{unicita_zakai}$.
\end{proof}

\begin{rem}
If we assume that $\forall h \in \H \setminus \{0\}$, $\lambda^h_t>0$ $\forall t \in [0,T]$, instead of $(\ref{pilambda_h-positiva})$, then the condition $(\ref{positiva_h})$ is satisfied. Moreover, if in addition
\begin{equation*}
\widetilde{Z}_t=\mathcal{E}\left( - \int_0^t \frac{b_1(s)}{\sigma_1(s)}\ud W^1_s+\int_0^t\sum_{h\in\H\setminus\{0\}}\left(\frac{1}{\lambda^h_s}-1\right)[\ud N^h_s-\lambda_s^h \ud s]\right)
\end{equation*}
 is a $(P,\mathcal{F}_t)$-martingale, then $P_0$ coincides with the restriction on $\mathcal{F}^Y_T$, of the probability measure $\widetilde{P}_0$, equivalent to $P$ over $\mathcal{F}_T$ defined by $\ds \widetilde{Z}_t=\left.\frac{\ud \widetilde{P}_0}{\ud P}\right|_{\mathcal{F}_t}$.

\end{rem}

\subsection{Observation dynamics driven by independent point processes}

In this second example we consider the case where the observation dynamics is driven by independent point processes $N^i$ where only their intensities $\lambda^i$ are not observable as in the Example 3.15 in \cite{cc2011}.

In order to recall the model, we suppose that there exists a finite set of measurable functions $K^i_1(t,y)\neq 0$ for all $(t,y)\in [0,T]\times \mathbb{R}$,  for $i=1,...,n$, such that
\begin{equation*}
	d^1(t,x,y):=\{\zeta\in Z: K_1(t,x,y;\zeta)\neq 0\}=\bigcup_{i=1}^n d^1_i(t,x,y)\;\;\;\textrm{and}\;\;\;d^1_i(t,x,y)\cap d^1_j(t,x,y)=\emptyset \;\;\;\forall i\neq j
\end{equation*}
 where $d^1_i(t,x,y):=\{\zeta\in Z:K_1(t,x,y;\zeta)=K_1^i(t,y)\}$. Then define $\ds D^i_t=d^1_i(t,X_{t^-},Y_{t^-})$.

\medskip

In this framework the dynamics of $Y$ is
\begin{equation}\label{dinamica_di_Y}
dY_t=b_1(t,X_t,Y_t)\ud t + \sigma_1(t,Y_t)\ud W^1_t+\sum_{i=1}^n K^i_1(t,Y_{t^-}) \ud N^i_t
\end{equation}
where $N^i_t=N((0,t]\times D^i_t)$, for $i=1,...,n$, and so $Y$ turns to be driven by independent counting processes. It is easy to see that the integer-valued measure $m(\ud t,\ud z)$ can be written as

\[
m(\ud t,\ud z)=\sum_{s:\Delta Y_s\neq 0}\delta_{\{s,\Delta Y_s\}}(\ud t,\ud z)=\sum_{i=1}^n \delta_{K^i_1(t,Y_{t^-})}(\ud z) \ud N^i_t .
\]

Let us define the functions $\lambda^i(t,x,y):=\nu(d^1_i(t,x,y))$ for $i=1,...,n$ then the $(P,\mathcal{F}_t)$-intensities of each point process $N^i$ is given by $\lambda^i_t=\nu(D^i_t)$.
\vskip 3mm

For this model the $(P,\mathcal{F}_t)$-dual predictable projection of the measure $m(\ud t,\ud z)$ is

\begin{equation*}
\lambda_t\phi_t(\ud z)\ud t=\int_{D_t}\delta_{K_1(t,\zeta)}(\ud z)\nu(\ud\zeta)\ud t
                     =\sum_{i=1}^n\delta_{K^i_1(t,Y_{t^-})}(\ud z)\int_{D^i_t}\nu(\ud \zeta)\ud t
                     =\sum_{i=1}^n\delta_{K^i_1(t,Y_{t^-})}(\ud z)\lambda_t^i\ud t,
\end{equation*}

and the $(P,\mathcal{F}^Y_t)$-dual predictable projection is
\[
\nu^p(\ud t,\ud z)=\pi_{t^-}(\lambda \phi (\ud z))\ud t=\sum_{i=1}^n\delta_{K^i_1(t,Y_{t^-})}(\ud z)\pi_{t^-}(\lambda^i)\ud t.
\]

\medskip

The Kushner-Stratonovich equation solved by $\pi$ in this special case  is given by

\begin{equation}\label{filtro}
\begin{split}
\pi_t(f)= f(0,x_0)+\int_0^t\pi_s(L^Xf)\ud s+\int_0^t\sigma_1(s)^{-1}[\pi_s(b_1 f)-\pi_s(b_1)\pi_s(f)]+\rho\pi_s\left(\sigma_0\frac{\partial f}{\partial x}\right)\ud I_s+\\
+\sum_{i=1}^n\int_0^t(\pi_{s^-}(\lambda^i))^+ \left(\pi_{s^-}(\lambda^i f) -\pi_{s^-}(f)\pi_{s^-}(\lambda^i) +\pi_{s^-}(R^i f)\right)\left(\ud N^i_s-\pi_{s^-}(\lambda^i)\ud s\right),
\end{split}
\end{equation}

where, $\forall i=1,...,n$, $\ds R^i$  is the operator
\[
R^i f(t,x,Y_{t^-})=\int_{d^1_i(t,x,Y_{t^-})}\left[  f(t,x+K_0(\zeta;t))-f(t,x)\right] \nu(\ud\zeta)
\]

that takes into account common jump times between the signal $X$ and the point process $N^i$.

\medskip

This result has been obtained in \cite{cc2011} under the hypothesis that  $\lambda^i_t>0$ $P-a.s.$,  $\forall i = 1,...,n$, $\forall t \in [0,T]$. Nevertheless this condition can be weakened and the same result can be proved under the hypothesis that $\pi_t(\lambda^i)>0$ $P-a.s.$,  $\forall i = 1,...,n$, $\forall t \in [0,T]$,  thanks to Theorem $\ref{EQ}$.

\medskip

Now, assume that for $i=1,...,n$,

\begin{equation}\label{positiva}
\pi_{t^-}(\lambda^i)>0 \quad P-a.s.\quad \forall t\in[0,T],
\end{equation}

then the Assumption \ref{assumption_eta} is fulfilled since the $(P, \mathcal{F}^Y_t)$-dual predictable projection of the measure $m(\ud t,\ud z)$ turns to be equivalent to
$$\ds \eta(t, Y_{t^-}, \ud z):=\sum_{i=1}^n \delta_{K^i_1(t, Y_{t^-})}(\ud z).$$

\medskip

Observe that the measure $\eta(t, Y_{t^-}, \ud z)$ satisfies the condition $(\ref{integrabilita_eta})$, in fact
\[
\sum_{i=1}^n \delta_{K^1(t, Y_{t^-})}(\mathbb{R})=n.
\]

\medskip
We recall that the probability measure $P_0$ defined over $(\Omega, \mathcal{F}^Y_T)$, is such that $\widetilde{W}^1_t:=I_t+\int_0^t\pi_s\left(\frac{b_1}{\sigma_1}\right)\ud s$ becomes a $(P_0, \mathcal{F}^Y_t)$-Brownian motion, and the integer valued random measure $m(\ud t, \ud z)$ has $(P_0,\mathcal{F}^Y_t)$-dual predictable projection $\eta(t,Y_{t^-}, \ud z)$.
With this choice, each point process $N^i$ becomes a $(P_0, \mathcal{F}^Y_t)$-standard Poisson process, that is with intensity equal to 1, and the $\ds (P_0,\mathcal{F}^Y_t)$-intensity of the point processes $\ds N_t=\sum_{i=1}^n N^i_t$ becomes equal to $n$. Hence $P_0$ is given by

\begin{equation}\label{P0}
\left.\frac{\ud P_0}{\ud P}\right|_{\mathcal{F}^Y_t}=Z_t=\mathcal{E}\left(-\int_0^t \pi_s\left(\frac{b_1}{\sigma_1}\right)\ud I_s+\sum_{i=1}^n \int_0^t\left(\frac{1}{\pi_{s^-}(\lambda^i)}-1\right)(\ud N^i_s - \pi_{s^-}(\lambda^i) \ud s)\right).
\end{equation}

\medskip

For this model, the Protter-Shimbo condition is analogous to $(\ref{shimbo_2})$, found for the previous model.

\medskip

Let us compute the dynamics of $\vr(1)$. Recall that $\vr_t(1)=Z_t^{-1}:=\left.\frac{\ud P}{\ud P_0}\right|_{\mathcal{F}^Y_t}$, then, by considering the effects of the Girsanov change of measure on the processes we get the following result.
Under the assumption that $Z$ is a $(P,\mathcal{F}^Y_t)$-martingale  and $(\ref{positiva})$, the process $\vr(1)$ solves
\[
\ud \vr_t(1)=\vr_{t^-}(1)\left[\pi_t\left(\frac{b_1}{\sigma_1}\right)\ud \widetilde{W}^1_t+\sum_{i=1}^n(\pi_{t^-}(\lambda^i)-1)(\ud N^i_t-\ud t)\right].
\]

\medskip

Finally, the dynamics of the unnormalized filter can be deduced by Theorem $\ref{thm_zakai}$.

\begin{corollary}[The Zakai equation]\
Assume the hypotheses of Theorem $\ref{EQ}$ and $(\ref{positiva})$.
Let $P_0$ be the probability measure defined by $(\ref{P0})$ and $Z$ a $(P,\mathcal{F}^Y_t)$-martingale. Assume moreover $(\ref{hp_deboli2})$,
then, $\forall f \in \mathcal{C}_b^{1,2}([0,T]\times \mathbb{R})$, the unnormalized filter $\vr$ satisfies the Zakai equation

\begin{equation}\label{zakai_ni}
\begin{split}
\ud \vr_t (f) =\vr_t(L^X f)\ud t+\left[ \frac{\vr_t(b_1 f)}{\sigma_1(t)}+ \rho \vr_t\left(\sigma_0 \frac{\partial f}{\partial x}\right)\right]\ud \widetilde{W}^1_t  +\sum_{i=1}^n\left[\vr_{t^-}(\lambda^i f)-\vr_{t^-}(f)+\vr_{t^-}(R^i f)\right](\ud N^i_t-\ud t).
\end{split}
\end{equation}
\end{corollary}

\medskip

We conclude with a strong uniqueness result, stated in the corollary below.

\begin{corollary}[Uniqueness for the Zakai equation]\label{processi_di _punto}
Let $(X,Y)$ be the usual partially observed system $(\ref{sistema})$, where in particular the dynamics of $Y$ is given by $(\ref{dinamica_di_Y})$. Assume that for $\quad i=1,...,n$,
\begin{equation}\label{positiva2}
\vr_{t^-}(\lambda^i)>0 \quad P-a.s. \quad \forall t\in[0,T],
 \end{equation}
 $(\ref{assumption_c})$ and either $(\ref{g1})$ or $(\ref{g2})$. Let $\xi$ be a strong solution to the Zakai such that
\begin{equation}\label{positiva_i} \xi_{t^-}(\lambda^i)>0 \quad P-a.s. \quad \forall t\in[0,T], \quad i=1,...,n,\end{equation}
then $\xi_t(f)=\vr_t(f)$ $P_0-a.s.$ $\forall t\in[0,T]$.
\end{corollary}

\begin{proof}
As in the proof of Corollary $\ref{zakai-uniq2}$, it is sufficient to verify that, under $(\ref{positiva2})$ and $(\ref{positiva_i})$, the measures $\ds \xi_{t^-}(\lambda \phi(\ud z))=\sum_{i=1}^n \xi_{t^-}(\lambda^i) \delta_{K^i_1(t, Y_{t^-})}(\ud z)$ and $\ds \vr_{t^-}(\lambda \phi(\ud z))=\sum_{i=1}^n\vr_{t^-}(\lambda^i) \delta_{K^i_1(t, Y_{t^-})}(\ud z)$ are equivalent. Then the result follows from Theorem $\ref{unicita_zakai}$.
\end{proof}

\subsection{State process given by a Pure Jump Process}

Suppose that $X$ is pure jump process taking values in a countable space $S$. For such a particular case, we will show directly that pathwise uniqueness for the solutions to the Zakai equation holds, since it can be solved recursively.

\medskip

We can represent the process $X$ as the solution to the following equation
\begin{equation}
X_t=X_0+\int_0^t\sum_{u\in S}\sum_{v\in S} \ind_{\{X_{s^-}=u\}}(v-u)N(\ud s; u,v ),
\end{equation}
here $\forall (u,v)\in S \times S$, $N(t; u,v)$ is a Poisson process with parameter  $\lambda_0(u)\mu_0(u,v)$, where $\lambda_0:S\rightarrow [0,\infty)$ is a measurable function and  $\ds \forall u$, $\ds \mu_0(u,\cdot)$  is a probability measure over $S$.
This notation corresponds to the choice
$\ds K_0(t,X_{t},\zeta)=\ind_{\{X_{t}=u\}}(v-u)$, $\zeta=(u,v)$ and $\nu( \{\zeta\})=\lambda_0(u)\mu_0(u, v)$.
\medskip

If we apply It\^{o}'s formula to the function $f\in\mathcal{C}^{1,2}([0,T]\times \mathbb{R})$, we can get
\begin{equation*}
\begin{split}
\ud f(t,X_t)=&\frac{\partial f}{\partial t}(t,X_t)\ud t + \sum_{u\in S}\sum_{v\in S} \ind_{\{X_{s^-}=u\}} \left\{f(t,v)-f(t,u)\right\} \left[N(\ud s; u,v )-\lambda_0(u)\mu_0(u,v)\ud t\right]\\
            & + \lambda_0(X_{t^-})\sum_{v\in S} \left\{f(t,v)-f(t,X_{t^-})\right\}\mu_0(X_{t^-},v)\ud t,
\end{split}
\end{equation*}

then the generator of the Markov process $X$ is

$$
L^X f(t,x)=\frac{\partial f}{\partial t}(t,x)+\lambda_0(x)\sum_{v\in S} \left\{f(t,v)-f(t,x)\right\}\mu_0(x,v).
$$

In particular the operator $L^X_0$ is given by

$$
L_0^X f(t,x)=\frac{\partial f}{\partial t}(t,x)+\lambda_0(x)\sum_{v\in S_1(x,y)} \left\{f(t,v)-f(t,x)\right\}\mu_0(x,v),
$$
where $S_1(x,y):=\left\{v \in S: v\neq x ,\quad K_1(t,x,y;x,v)= 0 \right\}$, and the operator $\overline{L}f(z)$, which takes into account common jump times between the processes $X$ and $Y$ is
$$
\overline{L}_tf (\ud z)=\overline{L} f(t,X_{t^-},Y_{t^-},\ud z)=\lambda_0(X_{t^-})\sum_{v\neq X_{t^-}} \left[f(t,v)-f(t,X_{t^-})\right]\mu_0(X_{t^-},v)\delta_{K_1(t, X_{t^-}, Y_{t^-}; X_{t^-},v)}(\ud z).
$$

In this situation the Zakai equation is given by

\begin{equation*}
\begin{aligned}[t]
&\vr_t(f)=\int_0^t \left[\vr_{s}(L^X_0f)- \vr_{s}(\lambda f)+\eta_s(\mathbb{R})\vr_{s}(f)\right]\ud s + \int_0^t \frac{\vr_s(b_1 f)}{\sigma_1(s)} \ud \widetilde{W}^1_s \\
&\qquad + \int_0^t\int_{\mathbb{R}}\left[\frac{\ud \vr_{s^-}(\lambda \phi f)}{\ud \eta_s}(z)-\vr_{s^-}(f)+\frac{\ud \vr_{s^-}(\overline{L}f)}{\ud \eta_s}(z)\right]m(\ud s, \ud z).
\end{aligned}
\end{equation*}


\medskip

\begin{proposition}
Let the state process $X$ be a purely jump process taking values in a finite space $S$. Assume existence and uniqueness for the solution of the system $(\ref{sistema})$, then pathwise uniqueness for the solution to the Zakai equation holds.

\end{proposition}

\begin{proof}
We prove that, whenever the space of the values of the process $X$ is finite, the Zakai equation can be computed recursively.

\medskip
Note that $\vr_t(f)=\sum_{u\in S}f(t,u)V(t,u)$, where $V(t,u):=\vr_t(f_u)$ and $f_u(x)=\ind_{\{x=u\}}$. For $t\in [T_n, T_{n+1})$, $\{ V(t,u) \}_{u\in S}$ solves
\begin{equation}\label{sistema1}
\begin{split}
V(t,u)= &\!\!\!\int_{T_n}^t \Big\{  [\eta_s(\mathbb{R})-\lambda_0(u)\mu_0(u,S_1(u,Y_s))-\lambda(s,u,Y_{s^-})]  V(s,u) + \!\!\!\sum_{v\in S_1(u,Y_{s^-})} \!\!\!\!\! \! \lambda_0(v)\mu_0(v,u)V(s,v) \Big\}\ud s +
 \\
 & \quad + \int_{T_n}^t\frac{b_1(s,u,Y_{s})}{\sigma_1(s,Y_s)}  V(s,u) \ud\widetilde{W}^1_s.
\end{split}
\end{equation}

 where, as usual, $\{T_n\}$ denotes the sequence of jump times of $Y$. Moreover, for  $t=T_n$
\[
\begin{aligned}[t]
&V(T_n,u)=\Big\{ \lambda(T_n,u,Y_{T_n^-}) \frac{\ud \phi_{T_n}(u)}{\ud \eta_{T_n}}(Z_n) -  1 +   \frac{\ud \beta_{T_n}(u)}{\ud \eta_{T_n}}(Z_n)  \Big \} V(T_n^-,u)\\
&\qquad \qquad + \sum_{v\in S} \frac{\ud \alpha_{T_n}(u, v)}{\ud \eta_{T_n}}(Z_n) V(T_n^-,v) \Big\},
\end{aligned}
\]

where

$$\beta_t(u, Y_{t^-}, \ud z) := \lambda_0(u) \sum_{v\in S} \mu_0(u,v)\delta_{K_1(t,  u, Y_{t^-};  u,v)}(\ud z), $$

$$\alpha_t(u, v, Y_{t^-}, \ud z) := \sum_{v\in S} \lambda_0(v)  \mu_0(v,u)\delta_{K_1(t,  v, Y_{t^-};  v,u)}(\ud z).$$

Hence $\{ V(T_n,u) \}_{u\in S}$  is completely determined by the observed data $(T_n, Z_n)$ and by the knowledge of $\{ V(t,u) \}_{u\in S}$  for all $t\in[T_{n-1}, T_n)$, since  $\forall u \in S$, $V(T_n^-, u)  = \lim_{t \to  T_n^-} V(t, u)$.
\medskip

Therefore, if $S$ is finite, the unnormalized filter can be computed by solving the system of linear stochastic differential equations (\ref{sistema1}) between two consecutive jump times of $Y$, which has a unique solution.

\end{proof}

\begin{rem}
If the state process $X$ and the observation process $Y$ have only common jump times, that is $S_1=\emptyset$, the equation between two consecutive jump times becomes

\[
V(t,u)= \!\!\!\int_{T_n}^t \! V(s,u)[\eta_s(\mathbb{R})-\lambda(s,u,Y_{s^-})] \ud s+ \int_{T_n}^t\frac{b_1(s,u,Y_{s})}{\sigma_1(s,Y_s)}V(s,u)\ud\widetilde{W}^1_s.
\]
Thus, as in the model considered in \cite{CG1}, the unnormalized filter can be computed recursively even if the set $S$ is countable.
\end{rem}

\appendix
\begin{center}
\Large{\textbf{Appendix}}
\end{center}



\section{}\label{appc}

Below we recall the statement of Proposition III.3.5 proved in \cite{JS}, that we used to prove that uniqueness for the Kushner-Stratonovich equation implies uniqueness for the Zakai one. This proposition supplies a useful property of the density process.

\begin{proposition}
Let $P$ and $P'$ be two probability measures over $(\Omega, \mathcal{F})$.
Assume that $P'\stackrel{loc}{<<}P$ and let $Z$ be the density process.
Then $\ds P'\left(\inf_{t\geq0}Z_t>0\right)=1$.
\end{proposition}

\section{}\label{appb}

\begin{proposition}\label{sup}
Let $\theta$ be the process defined in $(\ref{theta})$, where the process $\mu$ is a strong solution to the Kushner-Stratonovich equation such that $\mu_{t^-}(\lambda \phi(\ud z))$ is equivalent to the measure $\pi_{t^-}(\lambda \phi(\ud z))$ for every $t\in[0,T]$ and assume the hypotheses of Theorem $\ref{thm_zakai}$. Then $\sup_t \theta_t <\infty $ $P_0-a.s.$
\end{proposition}

\begin{proof}
Observe that the Radon-Nikodym derivative $\ds \frac{\ud \mu_{t^-}(\lambda \phi)}{\ud \eta_t}(z)$ is well defined; in fact the measures $\ds \mu_{t^-}(\lambda \phi (\ud z))$ and $\ds \eta(t,Y_{t^-}, \ud z)$ are equivalent since they are both equivalent to $ \pi_{t^-}(\lambda \phi (\ud z))$.

Therefore the process $\theta$ is the Dol\'{e}ans Dade exponential of the $(P_0, \mathcal{F}^Y_t)$-martingale
\begin{equation*}
M_t= \int_0^t \mu_s\left(\frac{b_1}{\sigma_1}\right)\ud \widetilde{W}^1_s + \int_0^t \int_{\R} \left(\frac{\ud\mu_{s^-}(\lambda\phi)}{\ud\eta_s }(z)-1\right) \left[m(\ud s, \ud z) - \eta_s (\ud z) \ud s\right].
\end{equation*}
Since $\ds \mu_{t^-}(\lambda \phi (\ud z))$ and $\ds \eta(t,Y_{t^-}, \ud z)$ are equivalent, i.e. $\ds \frac{\ud\mu_{t^-}(\lambda\phi)}{\ud\eta_t }(z)>0$ $P_0-a.s.$ $\forall (t,z) \in [0,T]\times \mathbb{R}$, then $\theta$ is a strictly positive $(P_0, \mathcal{F}^Y_t)$-local martingale, which means in particular that $\theta$ is a supermartingale with $\mathbb{E}^{P_0}[\theta_t]\leq \mathbb{E}^{P_0}[\theta_0]=1$, $\forall t \in [0,T]$.

Recall that, by the Markov inequality,
\begin{equation}
P_0\left(\sup_t \theta_t \geq N \right)\leq \frac{1}{N}\xrightarrow[N\to\infty]{}0
\end{equation}

Then
\begin{eqnarray}
P_0\left(\sup_t \theta_t = \infty\right)=P_0\left(\bigcap_{N\geq 1}    \left\{   \sup_t \theta_t \geq N  \right\}   \right) \label{prima}\nonumber\\
                       =   P_0\left(\lim_{N\to\infty}    \left\{   \sup_t \theta_t \geq N  \right\}   \right)  \label{seconda}\\
                       \leq  \lim_{N\to\infty}   P_0\left(  \sup_t \theta_t \geq N    \right)=0  \label{terza}
\end{eqnarray}

where $(\ref{seconda})$ comes from the fact that the sequence events $\ds A_N:= \left\{   \sup_t \theta_t \geq N  \right\}$ is decreasing, that is $A_N \subseteq A_{N-1}$, and $(\ref{terza})$ comes from the Fatou Lemma.

\end{proof}

\section{}\label{appendixC}

In this last part of the appendix, we will give sufficient conditions (see for instance \cite{CG5} and \cite{GS}) which ensure
strong existence and strong uniqueness  for  solutions to the system $(\ref{sistema})$.

\medskip

\begin{ass}\label{assumption_c}
\begin{description}
  \item[(i)] Let $b_0(t,x),b_1(t,x,y)$, $\sigma_0(t,x)$, and $\sigma_1(t,y)$ be jointly continuous functions of their arguments, and  $K_0(t,x;\zeta),K_1(t,x,y;\zeta)$ $\mathbb{R}-$valued, jointly continuous  functions in $(t,x,y)$.

  \item[(ii)] Suppose there exists a constant $C>0$ such that $\forall t \in [0,T]$

\begin{equation}\label{crescita}
\begin{split}
|b_0(t,x)|^2\leq C(1+|x|^2) ;  & \qquad |\sigma_0(t,x)|^2\leq C(1+|x|^2)\\
|b_1(t,x,y)|^2\leq C(1+|x|^2+|y|^2) ; & \qquad|\sigma_1(t,y)|^2\leq C(1+|y|^2)\\
\int_Z|K_0(t,x;\zeta)|^2\nu(\ud\zeta)\leq C(1+|x|^2);  & \qquad\int_Z|K_1(t,x,y;\zeta)|^2\nu(\ud\zeta)\leq C(1+|x|^2+|y|^2)
\end{split}
\end{equation}

\item[(iii)] $\forall r>0$, there exists a constant  $L=L(r)>0$ such that, $\forall x,x',y,y'\in B_r(0):=\{z \in \mathbb{R}:|z|\leq r\}$
\begin{equation}\label{lipschitzlocale}
\begin{split}
&\qquad|b_0(t,x)-b_0(t,x')|\leq L |x-x'| \qquad |\sigma_0(t,x)-\sigma_0(t,x')|\leq L |x-x'|\\
&\;|b_1(t,x,y)-b_1(t,x',y)|\leq L (|x-x'|+|y-y'|) \qquad |\sigma_1(t,y)-\sigma_1(t,y')|\leq L |y-y'|\\
&\int_Z|K_0(t,x;\zeta)-K_0(t,x';\zeta)|^2\nu(\ud\zeta) \leq L|x-x'|^2\\
&\int_Z|K_1(t,x,y;\zeta)-K_1(t,x',y';\zeta)|^2\nu(\ud\zeta)\leq L (|x-x'|^2+|y-y'|^2)
\end{split}
\end{equation}
\end{description}
\end{ass}
We refer to (\ref{crescita}) and (\ref{lipschitzlocale}) respectively as growth conditions and local Lipschitz conditions.
 \medskip

Other classes of conditions which imply strong existence and weak uniqueness of  solutions to system (\ref{sistema})
without requiring continuity of $K_i$, $i=0,1$, can be deduced by those given in
\cite{CG5}, Appendix A.

\medskip





\end{document}